\documentclass[microtype]{gtpart}
\agtart

\usepackage{stmaryrd,ifmtarg,enumitem,tikz-cd,braket,mathtools}
\usepackage[hyperpageref]{backref}
\usepackage{cleveref}

\title{Magnitude homology of enriched categories and metric spaces}

\author{Tom Leinster}
\givenname{Tom}
\surname{Leinster}
\address{School of Mathematics, University of Edinburgh}
\email{Tom.Leinster@ed.ac.uk}
\urladdr{https://www.maths.ed.ac.uk/~tl}
\author{Michael Shulman}
\givenname{Michael}
\surname{Shulman}
\address{Department of Mathematics, University of San Diego}
\email{Shulman@sandiego.edu}
\urladdr{http://home.sandiego.edu/~shulman}

\keyword{magnitude}
\keyword{magnitude homology}
\keyword{Hochschild homology}
\keyword{Euler characteristic}
\keyword{enriched category}
\keyword{metric space}
\keyword{categorification}
\subject{primary}{msc2020}{18G90}
\subject{secondary}{msc2020}{16E40}
\subject{secondary}{msc2020}{51F99}
\subject{secondary}{msc2020}{55N31}

\arxivreference{1711.00802}

\definecolor{darkgreen}{rgb}{0,0.45,0} 
\renewcommand*{\backref}[1]{}
  \renewcommand*{\backrefalt}[4]{({%
      \ifcase #1 Not cited.%
            \or on p.~#2%
            \else on pp.~#2%
      \fi%
    })}
\newtheorem{thm}{Theorem}[section]
\crefname{thm}{Theorem}{Theorems}
\newtheorem{lem}[thm]{Lemma}
\crefname{lem}{Lemma}{Lemmas}
\newtheorem{cor}[thm]{Corollary}
\crefname{cor}{Corollary}{Corollaries}
\theoremstyle{definition}
\newtheorem{defn}[thm]{Definition}
\crefname{defn}{Definition}{Definitions}
\newtheorem{eg}[thm]{Example}
\crefname{eg}{Example}{Examples}
\newtheorem{egs}[thm]{Examples}
\crefname{egs}{Examples}{Examples}
\newtheorem{assume}[thm]{Assumption}
\crefname{assume}{Assumption}{Assumptions}
\theoremstyle{remark}
\newtheorem{rmk}[thm]{Remark}
\crefname{rmk}{Remark}{Remarks}

\newcommand{\V}{\mathbf{V}}

\newcommand{\ob}{\mathrm{ob}}
\newcommand{\op}{^{\mathrm{op}}}
\newcommand{\size}{\#}
\newcommand{\mg}[2][]{\mathsf{Mag}_{#1}(#2)}
\newcommand{\MC}{\mathrm{MC}}
\newcommand{\MCS}{\mathrm{MC}^\Sigma}
\newcommand{\MDS}{\mathrm{MD}^\Sigma}
\newcommand{\uMCS}{\smash{\widetilde{\mathrm{MC}}}^\Sigma}
\newcommand{\uMDS}{\smash{\widetilde{\mathrm{MD}}}^\Sigma}
\newcommand{\rnk}{\mathrm{rk}}
\let\k\Bbbk
\newcommand{\N}{\mathbb{N}}
\newcommand{\Ch}{\mathrm{Ch}}
\newcommand{\Ab}{\mathrm{Ab}}
\newcommand{\FinSet}{\mathrm{FinSet}}
\newcommand{\A}{\mathbf{A}}

\newcommand{\HH}{\textit{HH}} 

\newcommand{\ser}[1]{(\!(#1)\!)}
\newcommand{\pser}[1]{\llbracket #1\rrbracket}

\newcommand{\adj}{\operatorname{adj}}
\newcommand{\maybe}[1]{\text{?`}\qquad #1 \qquad \text{?}}

\newcommand{\sm}{s}
\DeclareSymbolFont{bbold}{U}{bbold}{m}{n}
\DeclareSymbolFontAlphabet{\mathbbb}{bbold}

\newcommand{\done}{\ensuremath{\mathbbb{1}}}

\newcommand{\makeushort}[3]{%
	\setbox0=\hbox{$\mathsurround=0pt #2{#3}$}%
	\hbox to 1\wd0{\hss\underbar{\hbox to #1\wd0{\hss\box0\hss}}\hss}}

\newcommand{\chain}[1]{\langle #1\rangle}
\newcommand{\btw}[3]{#1 \preceq #2 \preceq #3}
\newcommand{\sbtw}[3]{#1 \prec #2 \prec #3}
\newcommand{\pto}{\rightharpoonup}
\newcommand{\im}{\mathrm{im}}
\newcommand{\stil}{\widetilde{s}}
\makeatletter
\newcommand{\hmag}[1][]{\@ifmtarg{#1}{H^{\Sigma}}{H^{\Sigma,#1}}}
\def\thmqedhere{\expandafter\csname\csname @currenvir\endcsname @qed\endcsname}
\let\c@equation\c@thm
\makeatother
\renewcommand{\theenumi}{(\arabic{enumi})}

\numberwithin{equation}{section}

\begin{document}

\begin{abstract}
  Magnitude is a numerical invariant of enriched categories, including in
  particular metric spaces as $[0,\infty)$-enriched categories.  We show
  that in many cases magnitude can be categorified to a homology theory for
  enriched categories, which we call magnitude homology (in fact, it is a
  special sort of Hochschild homology), whose graded Euler characteristic
  is the magnitude.  Magnitude homology of metric spaces generalizes the
  Hepworth--Willerton magnitude homology of graphs, and detects geometric
  information such as convexity.
\end{abstract}

\maketitle


\section{Introduction}
\label{sec:introduction}

Magnitude is a numerical invariant of enriched categories, introduced by
the first author in~\cite{leinster:ec-cat,leinster:magnitude}.  
See~\cite{lm:mag-survey} for an overview; here we summarize the definition.
If $\V$ is a monoidal category, a $\V$-enriched category (or ``$\V$-category'') $X$ has a set of ``objects'' along with hom-objects $X(x,y)\in \V$ and identity and composition maps $\done\to X(x,x)$ and $X(y,z) \otimes X(x,y) \to X(x,z)$ satisfying unit and associativity axioms.
To define magnitude, we require in addition a ``size function'' $\size : \ob(\V) \to \k$, which is a monoid homomorphism from isomorphism classes of objects of $\V$ to the multiplicative monoid of a (semi)ring $\k$.
Given a $\V$-category $X$ with finitely many objects, one then defines a matrix $Z_X$ over $\k$ with entries $\size(X(x,y))$, and the magnitude of $X$ is the sum of all the entries of the inverse matrix $Z_X^{-1}$ (if it exists).

This (perhaps odd-looking) definition is motivated by the fact that the \emph{Euler characteristic} of the nerve of a finite poset $X$ can be computed as the sum of all the values of its M\"{o}bius function, which are precisely the entries of $Z_X^{-1}$ when $X$ is regarded as a category enriched over the poset $\mathbf{2}=\{\bot,\top\}$ of truth values, with $\size(\bot)=0$ and $\size(\top)=1$.
More generally, the first author showed that magnitude coincides with Euler characteristic if $X$ is any ordinary category whose nerve contains finitely many nondegenerate simplices, with $\V=\FinSet$ and $\size=$ cardinality.
Thus, magnitude is a generalization of Euler characteristic.

One particularly interesting example of magnitude is when $\V=[0,\infty]$ with the opposite ordering (that is, there is at most one morphism $k\to \ell$, and there is one exactly when $k\ge \ell$) and the monoidal structure of addition, in which case Lawvere~\cite{lawvere:metric-spaces} showed that $\V$-categories can be identified with (extended quasi-pseudo-)metric spaces.
If we take $\size(d) = e^{-td}$ for a real number $t$ (a ``length scaling factor''), we obtain a 1-parameter family of magnitudes of finite metric spaces, which have since been shown to capture a good deal of geometric information~\cite{lm:mag-survey,bc:mag-cpteuc,gg:mag-domeuc}.

The Euler characteristic of a space, on the other hand, is a fairly coarse invariant.
One very important refinement of it is ordinary homology, an \emph{algebraic} invariant consisting of a sequence of abelian groups $H_n(X)$ of which the Euler characteristic is the alternating sum of ranks $\sum_n (-1)^n \,\rnk\, H_n(X)$.
Thus, it is natural to conjecture that magnitude is the alternating sum of ranks of some kind of \emph{magnitude homology} theory, which contains even more geometric information than the numerical magnitude.
In~\cite{hw:mag-hom-gr}, Hepworth and Willerton constructed such a homology theory for the special case of graphs, regarded as metric spaces with the shortest path metric (or equivalently as categories enriched over the sub-monoidal-category $\N \subseteq [0,\infty]$).

The purpose of the present paper is to generalize this homology theory to a large class of enriching categories $\V$, and in particular to arbitrary metric spaces.
Specifically, when $\V$ is a semicartesian monoidal category (i.e.\ the monoidal unit is the terminal object), and $\Sigma:\V\to\A$ is a strong symmetric monoidal functor to an abelian category, we will define the \textbf{magnitude homology} $\hmag_*(X)$ of any $\V$-category $X$.\footnote{In fact, $\hmag_*(X)$ can be identified with the Hochschild homology of $X$ with ``constant coefficients'', but we will make no use of that.}
We will then show that given any ``rank function'' $\rnk : \ob(\A) \to \k$, the composite $\rnk \circ \Sigma$ is a size $\size : \ob(\V)\to\k$, and any sufficiently finite $\V$-category $X$ has a magnitude that can be computed as the Euler characteristic of $\hmag_*(X)$.

This is very abstract and general, but if we unwind it explicitly in the
case of metric spaces we obtain a calculable algebraic invariant defined
using $\R$-graded chain complexes. Here we take only the first steps in
investigating what information is contained in the
magnitude homology of a metric space; but since the first appearance of this work
as a preprint, other authors have built on it to obtain more sophisticated results. (See the works cited at the start of section~\ref{sec:geo-met} and
in section~\ref{sec:open-problems}.)
We show, for instance, that $\hmag_1(X)=0$ if and only if $X$ is \emph{Menger convex}, i.e.\ for any two distinct points there is another point strictly between them.
In particular, this implies that a closed subset $X\subseteq \mathbb{R}^n$ satisfies $\hmag_1(X)=0$ if and only if it is convex in the usual sense.
The meaning of $\hmag_n$ for $n>1$ is less clear, but for instance $\hmag_2$ seems to tell us something about the \emph{non-uniqueness} of geodesics connecting pairs of distinct points.

We begin by recalling the notion of magnitude in \cref{sec:mag}.
In \cref{sec:met} we describe explicitly and concretely the case of primary interest, namely the magnitude homology of metric spaces; and in \cref{sec:geo-met} we give some preliminary geometric interpretation of it.
Then in \cref{sec:hmag} we generalize to magnitude homology of any suitable enriched category.
In \cref{sec:rk-chi} we study ranks and Euler characteristics abstractly, enabling us to prove in \cref{sec:categ} the basic result that this general kind of magnitude homology is indeed a categorification of the magnitude.
We conclude with some open questions in \cref{sec:open-problems}.

\paragraph*{Acknowledgments} 
The theory of magnitude homology of enriched categories and metric spaces
was largely developed on the $n$-Category Caf\'{e} blog.  We would like to
thank all the participants in this conversation; in particular, Richard
Williamson contributed some important insights leading to homotopy
invariance (Corollary~\ref{thm:hmag-adj-invar}) and the connection with
Hochschild homology, while Beno\^{i}t Jubin corrected
Definition~\ref{defn:4cuts} and Example~\ref{eg:trees}.  The first author
thanks Aaron Greenspan, Richard Hepworth, Emily Roff and Simon Willerton
for helpful conversations.  We would also like to thank the referee, whose
suggestions we believe made the paper much more readable.

The first author was supported by a Leverhulme Trust Research
Fellowship. The second author was sponsored by The United States Air Force
Research Laboratory under agreement number FA9550-15-1-0053 and and
FA9550-16-1-0292.  The U.S. Government is authorized to reproduce and
distribute reprints for Governmental purposes notwithstanding any copyright
notation thereon.  The views and conclusions contained herein are those of
the authors and should not be interpreted as necessarily representing the
official policies or endorsements, either expressed or implied, of the
United States Air Force Research Laboratory, the U.S. Government, or
Carnegie Mellon University.

\section{Magnitudes of enriched categories and metric spaces}
\label{sec:mag}

We begin by recalling the notion of magnitude of enriched categories
from~\cite{leinster:ec-cat,leinster:magnitude}, including a slight
enhancement of the usual magnitude of metric spaces. For background on
enriched categories, see~\cite{kelly:basic}.
Let $(\V,\otimes,\done)$ be a symmetric monoidal category, and $\k$ a semiring (i.e.\ a ring without additive inverses), related by the following:

\begin{defn}\label{defn:size}
  A \textbf{size} is a function $\size : \ob(\V) \to \k$ that is
  \begin{itemize}
  \item invariant under isomorphism: if $a\cong b$ then $\size a = \size b$, and
  \item multiplicative: $\size(\done) = 1$ and $\size(a\otimes b) = \size a \cdot \size b$.
  \end{itemize}
\end{defn}

Here and throughout this work, all (semi)rings are assumed commutative.

\begin{eg}\label{eg:size-card}
  If $\V = \FinSet$, we can take $\k=\N$ (or, in fact, any semiring at all, since $\N$ is the initial semiring) and $\size$ the cardinality.
\end{eg}

\begin{eg}\label{eg:exp-size}
  If $\V=[0,\infty]$ with the opposite ordering and monoidal structure $+$, we can take $\k=\R$ and $\size a = e^{-a}$.
  This is the traditional choice of a size for $[0,\infty]$, but we can also use $q^a$ for any positive real number $q$.

  Since $q^a = e^{-ta}$ for $t = -\ln q$, using a different value of $q$ with $0<q<1$ is equivalent to scaling all numbers $a\in[0,\infty]$ by a positive real factor first.
  This is the traditional approach to the \emph{magnitude function}, which considers a metric space together with all of its rescalings by positive real factors.
\end{eg}

\begin{eg}\label{eg:size-ultra}
  Let $\V=[0,\infty]$ with the same ordering as in \cref{eg:exp-size}, but the monoidal structure $\max$.
  We will write this as $[0,\infty]_{\max}$ to avoid confusion.
  In this case, it is shown in~\cite[\S 8]{meckes:mag-div} that the only sizes with $\k=\R$ (or indeed any field) are of the form
  \[ \size_{\le k}(\ell) =
    \begin{cases}
      1 &\quad\text{if } \ell \le k\\
      0 &\quad \text{otherwise}
    \end{cases}
    \qquad \text{and}\qquad
    \size_{< k}(\ell) =
    \begin{cases}
      1 &\quad\text{if } \ell < k\\
      0 &\quad \text{otherwise}.
    \end{cases}
  \]
\end{eg}

\begin{eg}\label{eg:genpoly}
  If $\V$ is essentially small, but otherwise arbitrary, we can use the ``monoid semiring'' $\N[\V] = \N[\ob(\V)/\mathord{\cong}]$ of the monoid of isomorphism classes of objects in $\V$.
  Thus the elements of $\N[\V]$ are formal $\N$-linear combinations of isomorphism classes $[v]$, where multiplication is $\N$-linear and $[u] \cdot [v] = [u\otimes v]$.
  This is the universal example: any other size on $\V$ factors uniquely through it.

  In particular, if $\V=[0,\infty]$, then the elements of $\N[\V]$ are formal $\N$-linear combinations of numbers in $[0,\infty]$.
  We might write such an element as
  \[ a_1 [\ell_1] + a_2 [\ell_2] + \cdots + a_n [\ell_n] \]
  but since the multiplication is defined by $[\ell_1]\cdot [\ell_2] = [\ell_1+\ell_2]$, it is more suggestive to write $[\ell]$ as $q^\ell$ for a formal variable $q$, since then $q^{\ell_1} \cdot q^{\ell_2} = q^{\ell_1+\ell_2}$ looks like the ordinary rule for multiplication of powers.
  This yields a representation of elements of $\N[\V]$ as \textbf{generalized polynomials}
  \[ a_1 q^{\ell_1} + a_2 q^{\ell_2} + \cdots + a_n q^{\ell_n} \]
  in which the exponents can belong to $[0,\infty]$, rather than $\N$ as in an ordinary polynomial.
  We write this semiring of generalized polynomials as $\N[q^{[0,\infty]}]$.

  Note that the traditional size function of \cref{eg:exp-size} factors through this universal size via the ``evaluation'' map $\N[q^{[0,\infty]}]\to\R$ that substitutes $e^{-1}$ (or, more generally, any other positive real number) for $q$.
  Thus, the universal size valued in $\N[q^{[0,\infty]}]$ carries all the information of the sizes $e^{-ta}$ for all values of $t$.
\end{eg}

The definition of magnitude involves the following matrix.
Usually (finite) matrices are defined to have \emph{ordered} rows and columns, but for our purposes it is more convenient to consider matrices whose rows and columns are indexed by arbitrary finite sets.
A \emph{square} matrix is one whose rows and columns are indexed by the \emph{same} finite set.
Categorically speaking, there is a category whose objects are finite sets and whose morphisms $A\to B$ are functions $A\times B \to \k$, with composition by matrix multiplication; the square matrices are the endomorphisms in this category.

\begin{defn}
  Let $X$ be a $\V$-category with finitely many objects.
  Its \textbf{zeta function} is the $\ob(X)\times \ob(X)$ matrix over $\k$ defined by
  \[ Z_X(x,y) = \size(X(x,y)). \]
\end{defn}

\begin{defn}[\cite{leinster:ec-cat,leinster:magnitude}]\label{defn:mag}
  We say that $X$ \textbf{has M\"{o}bius inversion} (with respect to $\k$ and $\size$) if $Z_X$ is invertible over $\k$.
  In this case, the \textbf{magnitude} of $X$ is the sum of all the entries of its inverse matrix $Z_X^{-1}$.
\end{defn}

Since magnitude generalizes Euler characteristic and cardinality, it is sometimes written $\chi(X)$ or $|X|$ or $\# X$.
However, we will use all of those notations for other things, so we will write the magnitude of $X$ as $\mg{X}$, or $\mg[\k]{X}$ or $\mg[\size]{X}$ if necessary to indicate the relevant semiring or size function.

\begin{eg}
  If $\V = \FinSet$ and $\size$ is the cardinality valued in $\Q$, then it is shown in~\cite{leinster:ec-cat} that if $X$ is a finite ordinary category that is skeletal and contains no nonidentity endomorphisms, then $X$ has M\"{o}bius inversion, and its magnitude is equal to the Euler characteristic of (the geometric realization of) its nerve.
\end{eg}

\begin{eg}\label{eg:mag-met}
  As noted by~\cite{lawvere:metric-spaces}, a metric space $X$ can be regarded as a $\V$-category for $\V=[0,\infty]$ as in \cref{eg:exp-size}, where the hom-objects $X(x,y)$ are the distances $d(x,y)$.
  In addition, $[0,\infty]$-categories also include the following generalizations of metric spaces:
  \begin{itemize}
  \item \emph{Pseudo-metric spaces}, which drop the ``separation'' requirement that $d(x,y)=0$ implies $x=y$ (though we still have $d(x,x)=0$ for all $x$).
  \item \emph{Quasi-metric spaces}, which drop the symmetry requirement $d(x,y)=d(y,x)$.
  \item \emph{Extended metric spaces}, which allow distances to take the value $\infty$.
  \end{itemize}
  An arbitrary $[0,\infty]$-category combines all of these generalizations, hence can be called an \emph{extended quasi-pseudo-metric space}.
  Note that a quasi-pseudo-metric space is precisely a $[0,\infty)$-category.
  Additionally, a pseudo-metric space is a metric space if and only if it is \emph{skeletal} in the category-theoretic sense (any two isomorphic objects are equal); but for a quasi-pseudo-metric space to be skeletal means only that $d(x,y)=d(y,x)=0$ implies $x=y$, so it may not be a quasi-metric space.

  With the family of $\R$-valued size functions $e^{-td}$ from \cref{eg:exp-size}, the resulting magnitude of an (extended quasi-pseudo-)metric space was defined in~\cite{leinster:magnitude} and has since been extensively studied; see e.g.~\cite{lm:mag-survey}.
\end{eg}

\begin{eg}\label{eg:mag-gph}
  If $\V= \N$ with the opposite ordering and monoidal structure $+$, a $\V$-category is a quasi-pseudo-metric space whose distances are all integers.
  For instance, this includes any graph with the shortest-path metric.
  Of course we can simply consider any such $X$ as a metric space and apply \cref{eg:mag-met}, but as we will see later this case has some special properties.
\end{eg}

\begin{eg}\label{eg:mag-ultra}
  If $\V=[0,\infty]_{\max}$, then a $\V$-category is an (extended quasi-\mbox{pseudo-)}\emph{ultrametric space}.
  The ``ultramagnitude'' of an ultrametric space with respect to the size $\size_{\le k}$ (resp.\ $\size_{<k}$) from \cref{eg:size-ultra} was completely characterized in~\cite[\S8]{meckes:mag-div}: it counts the number of closed (resp.\ open) balls of radius $k$ (which form a partition of $X$).
\end{eg}

\begin{eg}\label{eg:genrat}
  In general, there tend to be more invertible matrices over $\k$ if it is a ring or a field.
  Thus, if $\k$ is given as a semiring, it is natural to universally complete it to a ring or a field.

  In particular, the universal semirings of \cref{eg:genpoly} can easily be completed to rings by simply allowing integer coefficients instead of natural numbers, obtaining the monoid ring $\Z[\V]$.
  When $\V=[0,\infty]$ this yields a ring $\Z[q^{[0,\infty]}]$ of generalized polynomials with integer coefficients but exponents in $[0,\infty]$; whereas when $\V=\N$ as in \cref{eg:mag-gph} it yields the ordinary polynomial ring $\Z[q] = \Z[q^\N]$.

  The polynomial ring $\Z[q]$ is an integral domain, so we can then pass to its field of fractions $\Q(q)$ consisting of formal rational functions.
  But 
$\Z[q^{[0,\infty]}]$ contains zero divisors:
  \[q^\infty (1 - q^{\infty}) = q^\infty - q^{\infty+\infty} = q^\infty - q^\infty = 0.\]
  However, if we omit $\infty$ (thereby requiring all distances in our metric spaces to be finite, i.e.\ omitting the ``extended''), we do get an integral domain $\Z[q^{[0,\infty)}]$.
  Its field of fractions, which we write $\Q(q^{[0,\infty)})$ or $\Q(q^{\R})$, consists of \textbf{generalized rational functions}
  \[\frac{a_1 q^{\ell_1} + a_2 q^{\ell_2} + \cdots + a_n q^{\ell_n}}{b_1 q^{k_1} + b_2 q^{k_2}+ \cdots b_m q^{k_m}} \]
  in which $a_i,b_j \in \Q$ and $\ell_i,k_j \in \R$.
  (There is no extra generality in allowing $a_i,b_j\in \Q$ and $\ell_i,k_j \in \R$ versus $a_i,b_j\in \Z$ and $\ell_i,k_j \in [0,\infty)$, since we can always multiply top and bottom by $A q^{N}$ for a sufficiently large $A\in \Z$ and $N\in \R$.)
\end{eg}

Note that we can try to ``evaluate'' a generalized rational function at any positive real value for $q$, as we did for generalized polynomials, but the result might not be defined (if the denominator ends up being zero).
Thus, working over the field $\Q(q^{\R})$ of generalized rational functions is a little better-behaved even than considering all real values for $q$ together.
In particular, we have the following:

\begin{thm}\label{thm:mag-met}
  Any finite quasi-metric space has M\"{o}bius inversion over $\Q(q^{\R})$.
\end{thm}
\begin{proof}
  The field $\Q(q^{\R})$ can be made into an \emph{ordered} field by inheriting the order of $\Q$ and declaring the variable $q$ to be infinitesimal.
  This means ordering generalized polynomials lexicographically on their coefficients, starting with the smallest (i.e.\ most negative) exponents of $q$.

  Now the condition $d(x,x)=0$ of a metric space means the diagonal entries of $Z_X$ are all $q^0=1$.
  On the other hand, the fact that $d(x,y)>0$ if $x\neq y$ means that all the off-diagonal entries $q^{d(x,y)}$ are infinitesimal.
  It follows that the determinant of $Z_X$ is a sum of $1$ (the diagonal term) and a finite number of infinitesimals.
  Thus this determinant is positive, and in particular nonzero; so $Z_X$ is invertible.
\end{proof}

\begin{rmk}
  If $X$ is a metric space (i.e.\ its distances are symmetric), then $Z_X$ is even positive definite over $\Q(q^\R)$.
  This follows from the Levy--Desplanques theorem over the ordered field $\Q(q^\R)$, since $Z_X$ is \emph{strictly diagonally dominant}: $Z_X(x,x) > \sum_{y\neq x} |Z_X(x,y)|$.
  (The Levy--Desplanques theorem is usually stated only for real or complex matrices~\cite[Theorem 6.1.10]{hj:matrix-analysis}, but holds over any ordered field~\cite{2444419}.)
\end{rmk}

It follows that \emph{any finite quasi-metric space $X$ has a magnitude lying in $\Q(q^{\R})$.}
We refer to this generalized rational function $\mg[{\Q(q^{\R})}]{X}$ as the \textbf{universal magnitude} of $X$.
Since a generalized polynomial has finitely many zeros (for essentially the same reason as an ordinary polynomial: as $q\to\infty$ the highest-degree term is dominant), we can evaluate the universal magnitude at all but finitely many positive real numbers $q$, obtaining a partial real-analytic function $\R\pto\R$.
By writing $q^\ell = e^{\ell \ln q}$, we can even extend this to a complex-analytic function defined at all but finitely many points of the Riemann surface of the logarithm (the universal cover of $\C \setminus\{0\}$).
If all distances in $X$ are integers as in \cref{eg:mag-gph}, then $\mg[{\Q(q^{\R})}]{X}$ coincides with $\mg[{\Q(q)}]{X}$ and hence lies in $\Q(q)$.
Thus, in this case the magnitude is a rational function of $q$, and in particular extends to a single-valued meromorphic function on all of $\C$.

The traditional \textbf{magnitude function} of $X$ is obtained by evaluating the universal magnitude $\mg[{\Q(q^{\R})}]{X}$ at $q=e^{-t}$ for positive real $t$.
For complex $t$ we can regard $q = e^{-t}$ as a parametrization of the Riemann surface of the logarithm, thereby extending the magnitude function to a complex-analytic function $\C\pto\C$ with finitely many singularities.
Richard Hepworth has pointed out that this parametrization $q=e^{-t}$ is the same as that used in the Laplace transform, which may be relevant for relating our perspective to the analytic generalizations of magnitude to infinite metric spaces~\cite{meckes:posdef,meckes:mag-div}.

Finally, we recall that magnitude can be generalized using \emph{weightings}.
The following definitions and theorems are all from~\cite{leinster:ec-cat,leinster:magnitude}.

\begin{defn}
  A \textbf{weighting} on a finite $\V$-category $X$ is a function $w:\ob(X) \to \k$ such that $\sum_{y} \size(X(x,y)) \cdot w(y) = 1$ for all $x\in X$.
  A \textbf{coweighting} on $X$ is a weighting on $X\op$.
\end{defn}

\begin{thm}
  If $\k$ is a field, then a $\V$-category $X$ has M\"{o}bius inversion if and only if it has a unique weighting $w$, and if and only if it has a unique coweighting $v$, in which case $\mg X = \sum_x w(x) = \sum_x v(x)$.\qed
\end{thm}

\begin{thm}
  If a $\V$-category $X$ has both a weighting $w$ and a coweighting $v$, then $\sum_x w(x) = \sum_x v(x)$.\qed
\end{thm}

\begin{defn}\label{defn:mag2}
  A $\V$-category $X$ \textbf{has magnitude} if it has both a weighting $w$ and a coweighting $v$, in which case its \textbf{magnitude} is the common value of $\sum_x w(x)$ and $\sum_x v(x)$.
\end{defn}

One virtue of this generalized notion of magnitude is that it is nontrivially invariant under equivalence of $\V$-categories.
A $\V$-category can only have M\"{o}bius inversion if it is skeletal, since two distinct isomorphic objects would produce two identical rows in $Z_X$.
But weightings and coweightings do transfer across equivalences:

\begin{thm}[{\cite[Lemma 1.12]{leinster:ec-cat} and~\cite[Proposition 1.4.1]{leinster:magnitude}}]\label{thm:mag-eqv0}
  If $X$ and $X'$ are equivalent $\V$-categories, and $X$ has a weighting, a coweighting, or has magnitude, then so does $X'$.
\end{thm}
\begin{proof}
  In the cited references this is proven under the assumption that all natural numbers are invertible in $\k$, so that the total weight on one isomorphism class in $X$ can be divided equally among all objects in the corresponding isomorphism class of $X'$.
  But this is unnecessary: we can simply choose one representative of the latter isomorphism class to give all the weight to.
\end{proof}

\begin{thm}[{\cite[Proposition 2.4]{leinster:ec-cat} and~\cite[Proposition 1.4.1]{leinster:magnitude}}]\label{thm:mag-eqv}
  If $X$ and $X'$ are equivalent and both have magnitude, then $\mg X = \mg{X'}$.\qed
\end{thm}

Not every category with magnitude in the sense of \cref{defn:mag2} is even equivalent to one with M\"{o}bius inversion (\cref{defn:mag}).
For instance,~\cite[Examples 4.3 and 4.5]{bl:ec-div} are skeletal categories having magnitude but not M\"{o}bius inversion.
However, we will not be very concerned with such examples in this paper, as our main criterion for relating magnitude homology to magnitude implies that the category has M\"{o}bius inversion (\cref{thm:hmag}).

\section{Magnitude homology of metric spaces}
\label{sec:met}

In this section we give a concrete description of the magnitude homology of a metric space and state its relationship to the magnitude.
A full proof of the latter will be obtained in \cref{sec:categ} as a consequence of a more general theory.

Let $X$ be a metric space; our goal is to describe a homology theory $H_*(X)$ such that when $X$ is finite, the alternating sum of ranks $\sum_n (-1)^n\,\rnk\, H_n(X)$ recovers its magnitude.
However, this idea presents us with two problems right from the start.
Firstly, $\mg{X}$ may not be an integer, whereas the rank of an abelian group is always an integer and hence so is an alternating sum of such.
Secondly, since the magnitude \emph{function} of $X$ is more geometrically meaningful than the bare numerical magnitude, we would like to recover it as well; but how can an alternating sum of ranks yield a real or complex function?

The solution to both problems is to aim, not for the magnitude or magnitude function, but the universal magnitude $\mg[{\Q(q^{\R})}]{X}$, or more precisely for a sort of ``asymptotic expansion'' of it.
Recall that the ring $\Q(q)$ of ordinary rational functions (which we might write $\Q(q^\Z)$ for consistency) can be embedded in the field $\Q\ser q$ of formal Laurent series, essentially by performing polynomial long division.
In~\cite{hw:mag-hom-gr} this was used to categorify the magnitude of a \emph{graph} by noting that when $\mg[{\Q(q)}]{X}$ is embedded into $\Q\ser q$ it lands in the subring $\Z\pser{q}$ of power series with integer coefficients, and then identifying these coefficients as alternating sums of the ranks of an $(\N,\N)$-bigraded homology theory.
That is, for a graph $X$ we have $\mg[{\Q(q)}]{X} = \sum_{\ell\in\N} a_\ell q^\ell$ in $\Q\ser q$, where $a_\ell = \sum_{n\in\N} (-1)^n \,\rnk\, H_{n,\ell}(X)$.

We can do something similar for arbitrary metric spaces $X$ by analogously embedding the field $\Q(q^{\R})$ of generalized rational functions in a field of ``generalized power series''.
Intuitively, such generalized power series should be expressions such as $\sum_{\ell\in\R} a_\ell q^\ell$, which can be represented formally by the coefficients function $a : \R\to\Q$.
Such expressions can always be added pointwise, and we would expect to be able to multiply them by the usual Cauchy product:
\[ \left(\sum_\ell a_\ell q^\ell\right) \left(\sum_\ell b_\ell q^\ell\right) = \sum_\ell \left(\sum_{j+k=\ell} a_j b_k\right) q^\ell. \]
The problem with this is that when we allow the exponents to be arbitrary real numbers, the sum $\sum_{j+k=\ell} a_j b_k$ is no longer finite.
(In fact this problem already arises for bi-infinite Laurent series $\sum_{\ell\in\Z} a_\ell q^\ell$, but can also occur with fractional exponents that are all positive.)
Thus, we need to impose some restriction on the \textbf{supports} $\Set{ \ell\in\R | a_\ell\neq 0}$ of our generalized power series.

The literature contains many kinds of generalized power series.
In a preprint version of this paper we used \emph{Hahn series}, which require that all such supports be well-ordered.
This suffices to ensure the Cauchy product is well-defined, but the field of Hahn series is much larger than necessary, and ill-behaved in certain formal ways (e.g.\ a Hahn series may not be the limit of its partial sums in the natural topology of the Hahn series field).
A first attempt at a smaller field might be the set of series of order type $\le\omega$, but these are not even closed under multiplication.
Thus, we will instead use \emph{Novikov series}~\cite{novikov:multivalued}; these happen to be the topological closure of $\Q(q^\R)$ in the Hahn series field, but we will give an explicit description.

\begin{defn}\label{defn:novikov-series}
  A \textbf{Novikov series} (with rational coefficients and real exponents)
  is a function $a:\R\to\Q$, written as a formal sum $\sum_{\ell\in\R}
  a_\ell q^\ell$, whose support is \textbf{left-finite}, i.e.\ for any
  $N\in \R$ the set $\Set{\ell < N | a_\ell \neq 0}$ is finite.
  We write $\Q\ser{q^\R}$ for the set of Novikov series.
\end{defn}

In the next theorem we compile the basic properties of $\Q\ser{q^\R}$.
Recall that a \emph{valuation} on a field $\k$ is a group homomorphism $\nu : \k^\times \to \Gamma$ from the multiplicative group of $\k$ to a totally ordered abelian group $\Gamma$ such that $\nu(a+b) \ge \min(\nu(a),\nu(b))$ with equality if $\nu(a)\neq\nu(b)$.
If $\Gamma$ is a subgroup of $\R$, then such a valuation induces a metric on $\k$ by $d(a,b) = e^{-\nu(a-b)}$; if $\k$ is complete in this metric it is called a \emph{complete valued field}.
Finally, an ordered field is called \emph{non-Archimedean} if it contains \emph{infinitesimals}: elements $a$ such that $0<a<\frac{1}{n}$ for all $n\in\N$.

\begin{thm}\label{thm:novikov-properties}
  $\Q\ser{q^\R}$ is a non-Archimedean ordered field under pointwise addition and the Cauchy product.
  It is a complete valued field with a valuation in $\R$, in which any Novikov series is the limit of its partial sums.
  Moreover, $\Q(q^\R)$ embeds in $\Q\ser{q^\R}$, and the latter is the Cauchy completion of the former in its induced valuation metric.
\end{thm}
\begin{proof}
  Left-finiteness ensures that the sums $\sum_{j+k=\ell} a_j b_k$ in the Cauchy product are finite, and that the Cauchy product itself again has left-finite support.
  Verification of the ring axioms is entirely straightforward, and essentially analogous to ordinary formal Laurent series (indeed, formal Laurent series with rational coefficients are precisely the Novikov series with integer exponents).

  To construct multiplicative inverses, first note that by multiplying by some $q^k$ it suffices to invert Novikov series of the form $\sum_{\ell\ge 0} a_\ell q^\ell$ where $a_0 \neq 0$.
  If the reciprocal of such a series is $\sum_{\ell\ge 0} b_\ell q^\ell$, then equating the Cauchy product to $1$ gives us a recursive formula for the $b_\ell$'s just as for ordinary formal power series:
  \begin{equation*}
    b_0 = \frac{1}{a_0} \hspace{2cm}
    b_\ell = - \frac{1}{a_0} \sum_{j+k=\ell\atop j>0} a_j b_k
  \end{equation*}
  This is well-defined once we specify a left-finite support for $b$, which we can take to be the set of all $\N$-linear combinations of exponents in the support of $a$.

  The valuation $\nu(a)$ of a Novikov series $a$ is the smallest exponent with a nonzero coefficient; the proof is just like the $\Z$-valuation of formal Laurent series.
  We define $a>0$ if $a_{\nu(a)}>0$; thus positive powers of $q$ are infinitesimal.

  As noted above, the induced metric is $d(a,b) = e^{-\nu(a-b)}$.
  (Note that unlike for formal Laurent series, the resulting metric topology is \emph{not} the ``adic'' topology induced by the valuation ideal of series with $\nu(a)>0$; indeed the latter ideal is idempotent.)
  A sequence $(a_n)$ in $\Q\ser{q^\R}$ is Cauchy if for any $k\in\R$ there is an $N\in \N$ such that for all $\ell<k$ the coefficients $(a_n)_\ell$ stabilize for $n>N$.
  These stable coefficients then form a Novikov series that is its limit, so $\Q\ser{q^\R}$ is complete.
  Since the coefficients of the partial sums of a given Novikov series $a$ obviously stabilize, they converge to $a$ in this topology.

  The integral domain $\Z[q^{[0,\infty)}]$ of generalized polynomials clearly embeds in $\Q\ser{q^\R}$, and hence so does its field of fractions $\Q(q^\R)$.
  Thus there is an induced valuation and metric on $\Q(q^\R)$.
  Finally, the partial sums of any Novikov series are in $\Q(q^\R)$, so $\Q(q^\R)$ is dense in $\Q\ser{q^\R}$; thus the latter is its Cauchy completion.
\end{proof}

In particular, we can also regard the universal magnitude $\mg[{\Q(q^\R)}]{X}$ as lying in $\Q\ser{q^\R}$.
Our goal, therefore, is to find an $(\N,\R)$-bigraded homology theory $H_{*,*}(X)$ associated to a metric space $X$, such that if $\mg[{\Q\ser{q^\R}}]{X} = \sum_{\ell\in\R} a_\ell q^\ell$ then $a_\ell = \sum_{n\in\N} (-1)^n \,\rnk\, H_{n,\ell}(X)$.
This is similar enough to the situation of~\cite{hw:mag-hom-gr} that we can essentially copy their explicit definition of the magnitude chain complex, simply allowing $\ell$ to range over real numbers rather than integers.

By a \emph{chain complex} we will always mean a homologically graded chain complex in nonnegative degrees, i.e.\ a sequence of objects $\{C_n\}_{n\in\N}$ of some abelian category with differentials $d_n:C_n \to C_{n-1}$ such that $d_{n-1}\circ d_n = 0$.

\begin{defn}\label{defn:mc}
  Let $X$ be a metric space.
  Its \textbf{magnitude complex} is an $\R$-graded chain complex defined as follows: $\MC_{n,\ell}(X)$ is the free abelian group generated by symbols $\chain{x_0,\dots,x_n}$, where each $x_i\in X$, such that $d(x_0,x_1)+\cdots+d(x_{n-1},x_n) = \ell$ and each $x_i \neq x_{i+1}$:
  \[ \MC_{n,\ell}(X) = \Z\Big[ \{ \chain{x_0,\dots,x_n} \mid d(x_0,x_1)+\cdots+d(x_{n-1},x_n) = \ell \text{ and } \forall i, x_i \neq x_{i+1} \} \Big]. \]
  The boundary map $d_n : \MC_{n,\ell}(X) \to \MC_{n-1,\ell}(X)$ is an alternating sum
  \[ d_n = \sum_{i=0}^n (-1)^i d_n^i \]
  where $d_n^i$, for $0<i<n$, discards the $i^{\mathrm{th}}$ point as long as this doesn't change the total distance:
  \[ \hspace{-1cm} d^i_n(\chain{x_0,\dots,x_n}) =
  \begin{cases}
    \chain{x_0,\dots,x_{i-1},x_{i+1},\dots,x_n} &  \quad\text{if } d(x_{i-1},x_i) + d(x_i,x_{i+1}) = d(x_{i-1}, x_{i+1})\\
    0 &\quad \text{otherwise},
  \end{cases}\hspace{-1cm}
  \]
  and $d^0_n = d^n_n = 0$.\footnote{We include $i=0$ and $i=n$ in the definition of $d$ even though $d^0_n=d^n_n=0$ for metric spaces, because in the more general context of \cref{sec:hmag} there are a $d^0_n$ and $d^n_n$ that need not vanish.}
\end{defn}

We leave it to the reader to verify that this is a chain complex, i.e.\ that $d\circ d = 0$.
Note that if $X$ is a graph with the shortest path metric, then $\MC_{n,\ell}(X)=0$ unless $\ell\in\N$, and its nonzero values are precisely the magnitude chain complex of a graph as defined in~\cite{hw:mag-hom-gr}.

\begin{defn}\label{defn:hmag-met}
  The \textbf{magnitude homology} $H_{*,*}(X)$ of $X$ is the homology of the $\R$-graded chain complex $\MC_{*,*}(X)$.
\end{defn}

In \cref{sec:hmag} we will describe a more general notion of magnitude homology for enriched categories, and in \cref{sec:categ} show that it categorifies the magnitude in a natural way.
Specialized to metric spaces, this will yield the following result:

\begin{thm}[\cref{thm:magcat-met-2}]\label{thm:magcat-met}
  If $X$ is a finite quasi-metric space, then
  \begin{align*}
    \mg[\Q\ser{q^\R}]{X} &= \sum_{n=0}^\infty (-1)^n \Big( \sum_{\ell\in\R} \rnk\, H_{n,\ell}(X) \, q^\ell \Big)\\
    &= \sum_{\ell\in\R} \Big( \sum_{n=0}^\infty (-1)^n \,\rnk\, H_{n,\ell}(X)\Big) q^\ell,
  \end{align*}
  the former infinite sum converging in the topology of $\Q\ser{q^\R}$.
\end{thm}

First, however, we briefly explore the geometric meaning of the magnitude homology for metric spaces.

\section{The geometric interpretation of magnitude homology}
\label{sec:geo-met}

In this section we will describe a few basic geometric properties of metric spaces that are detected by the magnitude homology.
For more recent progress in this area see~\cite{ky:maghom-ordcplx,jubin:maghom,gomi:magcplx,gomi:maghom-geodesic,asao:maghom-geodesic}.

Let $X$ be a metric space. First of all, in \cref{defn:mc} we have $d^0_n =
d^n_n = 0$, and in particular $d_1 = d^0_1 - d^1_1 = 0$, so there are no
0-boundaries and we have:

\begin{thm}
  $H_{0,0}(X)$ is the free abelian group on the points of $X$, and for $\ell>0$ we have $H_{0,\ell}(X)=0$.\qed
\end{thm}

We can also completely describe $H_{1,*}(X)$.

\begin{defn}
  Let $x,y,z\in X$.
  \begin{itemize}
  \item If $d(x,y)+d(y,z)=d(x,z)$, we say that $y$ is \textbf{between} $x$ and $z$. 
  \item If in addition $x\neq y\neq z$, we say $y$ is \textbf{strictly between} $x$ and $z$. 
  \end{itemize}
\end{defn}

\begin{thm}
  The group $H_{1,\ell}(X)$ is the free abelian group on the set of ordered pairs $\chain{x_0,x_1}$ such that $x_0\neq x_1$ and $d(x_0,x_1)=\ell$ and there does not exist any point strictly between $x_0$ and $x_1$.
\end{thm}
\begin{proof}
  The chain group $\MC_{1,\ell}(X)$ is freely generated by all pairs $\chain{x_0,x_1}$ such that $x_0\neq x_1$ and $d(x_0,x_1)=\ell$.
  Since $d_1 = 0$, all such chains are cycles.

  The chain group $\MC_{2,\ell}(X)$ is freely generated by triples $\chain{x_0,x_1,x_2}$ such that $x_0\neq x_1\neq x_2$ and $d(x_0,x_1) + d(x_1,x_2)=\ell$.
  We have $d_2 = -d^1_2$, so the boundary of $\chain{x_0,x_1,x_2}$ is $-\chain{x_0,x_2}$ if $x_1$ is strictly between $x_0$ and $x_2$, and $0$ otherwise.
  Thus, $\chain{x_0,x_2}$ is a boundary just when there is a point strictly between $x_0$ and $x_2$.
\end{proof}

In particular, the complete vanishing of $H_{1,*}(X)$ at all gradings has the following characterization.

\begin{defn}
  Two points $x,y\in X$ are \textbf{non-adjacent} if there exists a point strictly between them, and \textbf{adjacent} otherwise.
  The metric space $X$ is \textbf{Menger convex} if any two distinct points are non-adjacent.
\end{defn}

The term ``Menger convex'' is standard \cite[p.5]{papadopoulos:metricspaces}.
The term ``(non-)adjacent'' is not standard, but it is a faithful extension of the corresponding terminology for graphs.
From this perspective, note that a Menger convex metric space is ``as far from being a graph as possible'', since the shortest-path metric on a graph is \emph{defined} in terms of the adjacent pairs of points.

\begin{cor}\label{cor:menger-H1}
  $H_{1,\ell}(X)$ is freely generated by the ordered pairs of distinct adjacent points of $X$ at distance $\ell$ apart.
  In particular, $H_{1,*}(X)=0$ if and only if $X$ is Menger convex.\qed
\end{cor}

\begin{eg}
  If $X$ is a connected graph with the shortest path metric, then any pair of points at distance $>1$ apart will have a third point between them.
  Thus $H_{1,\ell}(X)$ vanishes except when $\ell=1$, in which case it is free on the pairs $\chain{x_0,x_1}$ such that $d(x_0,x_1)=1$, i.e.\ the oriented edges of $X$.
  This was noted in~\cite[Proposition 9]{hw:mag-hom-gr}.
\end{eg}

Menger convexity may seem a fairly weak condition, but in many cases it is equivalent to a more familiar strong sort of convexity.

\begin{defn}\label{defn:geodesic}
  A metric space $X$ is \textbf{geodesic} if for any points $x,y$ there is an isometry $\gamma :[0,a] \to X$ with $\gamma(0)=x$ and $\gamma(a)=y$ (hence $a=d(x,y)$).
\end{defn}

\begin{thm}
  If a metric space $X$ has the property that closed and bounded subsets of $X$ are compact, then $X$ is Menger convex if and only if it is geodesic.
\end{thm}
\begin{proof}
  See for instance~\cite[Theorem 2.6.2]{papadopoulos:metricspaces}.
\end{proof}

\begin{cor}\label{cor:closed-menger}
  A closed subset of $\R^n$ is Menger convex if and only if it is convex in the usual sense.\qed
\end{cor}

\begin{cor}
  If $X$ is a closed convex subset of $\R^n$, then $H_{1,*}(X)=0$.\qed
\end{cor}

On the other hand, any \emph{open} subset $X\subseteq \R^n$ is Menger
convex, since the straight line between two points of $X$ must intersect
the open balls around each of them that are contained in $X$. Hence an open
set $X \subseteq \R^n$ has trivial first homology. However, its closure
$\overline{X}$ need not be convex, and then $H_{1,*}(\overline{X}) \neq 0$
by Corollaries~\ref{cor:menger-H1} and~\ref{cor:closed-menger}.  Since the
closure of an open subset of $\R^n$ is also its Cauchy completion, $H_1$ is
not invariant under Cauchy completion.

If $H_{1,*}(X)$ fails to vanish completely, then its size tells us ``how badly'' $X$ fails to be Menger convex, and the gradings in which it fails to vanish tell us at what ``length scales'' this happens.

\begin{eg}
  Let $X$ be a closed annulus in the plane with inner diameter $\delta$.
  For distinct $x_0,x_1\in X$ there is a point strictly between them unless $x_0$ and $x_1$ are both on the inner boundary.
  The maximum distance between two points on the inner boundary is $\delta$, so $H_{1,\ell}(X)=0$ if $\ell>\delta$.
  If $0<\ell<\delta$, then for any $x_0$ on the inner boundary there are exactly two points $x_1$ on the inner boundary at distance $\ell$, whereas if $\ell=\delta$ there is exactly one (the antipodal point).
  Thus we have:
  \[ H_{1,\ell}(X) =
  \begin{cases}
    \Z[S^1\cdot 2] &\qquad 0<\ell<\delta\\
    \Z[S^1] &\qquad \ell=\delta\\
    0 &\qquad \ell>\delta
  \end{cases}
  \]
\end{eg}

\begin{eg}
  Let $X=X_1\sqcup X_2$ consist of two disjoint closed convex sets in
  $\R^n$ at a distance $\delta$ apart.
  Then if $x_0,x_1$ are both in $X_1$ or both in $X_2$, there is a point strictly between them; whereas if they are in different components then $d(x_0,x_1) \ge \delta$.
  Thus, $H_{1,\ell}(X)=0$ for $\ell<\delta$.
  However, for a suitable choice of $X_1$ and $X_2$ we can arrange that $H_{1,\ell}(X)\neq 0$ for arbitrarily large $\ell$, or even for all $\ell\ge\delta$, such as if $X_1$ is a line and $X_2$ is a point not on that line.
\end{eg}

Note that the previous two examples show that $H_{1,\ell}$ can vanish for
all large $\ell$ but not all small $\ell$, or for all small $\ell$ but not for all large $\ell$.

The geometric meaning of $H_{n,*}(X)$ for $n>1$ is not as obvious, but we
can get some idea by looking at $n=2$.  Let us introduce some more
terminology.

We write $\btw xyz$ to mean that $y$ is between $x$ and $z$, and $\sbtw xyz$ to mean that $y$ is strictly between $x$ and $z$.
In a general metric space, these notations are fundamentally ternary; but
in familiar spaces like $\R^n$, any of the following pairs of ternary conditions ensure that four points $x,y_1,y_2,z$ are collinear in that order.
\begin{enumerate}
\item $\btw x{y_1}{y_2}$ and $\btw x{y_2} z$.\label{item:btw1}
\item $\btw x{y_1}z$ and $\btw {y_1}{y_2} z$.\label{item:btw3}
\item $\btw x{y_1}{y_2}$ and $\btw {y_1}{y_2} z$.\label{item:btw2}
\end{enumerate}
In a general metric space, we can say the following:

\begin{lem}
  In a metric space $X$, conditions~\ref{item:btw1} and~\ref{item:btw3} above are equivalent, and both imply~\ref{item:btw2}.
\end{lem}
\begin{proof}
  Without loss of generality, suppose~\ref{item:btw1}.
  Then
  \begin{align*}
    d(x,y_1) + d(y_1,y_2) + d(y_2,z) &= d(x,y_2)+d(y_2,z)\\
    &= d(x,z).
  \end{align*}
  Therefore, using the triangle inequality, we have
  \begin{align*}
    d(y_1,z) &\ge d(x,z) - d(x,y_1)\\
    &= d(y_1,y_2) + d(y_2,z)\\
    &\ge d(y_1,z).
  \end{align*}
  Hence both inequalities are equalities, i.e.~\ref{item:btw3} holds.
  Finally, it is evident that once both~\ref{item:btw1} and~\ref{item:btw3} hold then~\ref{item:btw2} does.
\end{proof}

\begin{defn}\label{defn:4cuts}
  A metric space \textbf{has no 4-cuts} if whenever $y_1\neq y_2$, condition~\ref{item:btw2} above implies~\ref{item:btw1} and~\ref{item:btw3}, or equivalently whenever $y_1\neq y_2$, if $d(x,y_1)+d(y_1,y_2)=d(x,y_2)$ and $d(y_1,y_2)+d(y_2,z)=d(y_1,z)$ then $d(x,z) = d(x,y_1)+d(y_1,y_2)+d(y_2,z)$.
\end{defn}

\begin{eg}
  Of course, $\R^n$ has no 4-cuts, and the property of having no 4-cuts is inherited by subspaces.
\end{eg}

\begin{eg}
  For an example of a metric space that does have 4-cuts, consider the 4-cycle graph
  
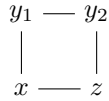
\[
  \begin{tikzpicture}
    \node (x) at (0,0) {$x$};
    \node (y1) at (0,1) {$y_1$};
    \node (y2) at (1,1) {$y_2$};
    \node (z) at (1,0) {$z$};
    \draw (x) -- (y1) -- (y2) -- (z) -- (x);
  \end{tikzpicture}
  \]
  with the shortest path metric.
  More generally, any graph containing a 4-cycle as a full subgraph (a.k.a.\ induced subgraph: the graph determined by some four vertices and all the edges between them) has 4-cuts.

  Yuzhou Gu has pointed out that this implication is not reversible: there are graphs that do not contain a 4-cycle as a full subgraph yet nevertheless have 4-cuts, such as the following.
  \[
    \begin{tikzpicture}
      \node (a) at (.5,-.7) {$x$};
      \node (b) at (0,.4) {$y_1$};
      \node (c) at (1,1) {$y_2$};
      \node (d) at (2,.4) {$z$};
      \node (e) at (1.5,-.7) {$w$};
      \draw (a) -- (b) -- (c) -- (d) -- (e);
      \draw (a) -- (e); \draw (b) -- (e); \draw (c) -- (e);
    \end{tikzpicture}
  \]
  Here there is no full subgraph that is a 4-cycle, yet $d(x,y_1)+d(y_1,y_2)=2=d(x,y_2)$ and $d(y_1,y_2)+d(y_2,z)=2=d(y_1,z)$, while $d(x,z) = 2$ whereas $ d(x,y_1)+d(y_1,y_2)+d(y_2,z) = 3$.
  In~\cite{gu:grmaghom-algmorse} Gu proves that the graphs without 4-cuts are precisely the ``Ptolemaic'' graphs.
\end{eg}

\begin{eg}\label{eg:trees}
  A tree has no 4-cuts.
  To prove this, note that in a tree there is exactly one path between any two vertices that does not visit any vertex twice, and this is also the unique path of shortest length.
  Now if $\btw{x}{y_1}{y_2}$ and $\btw{y_1}{y_2}{z}$ with $y_1\neq y_2$, we claim that following the shortest path from $x$ to $y_1$ followed by the shortest path from $y_1$ to $y_2$ and then the shortest path from $y_2$ to $z$ gives the shortest path from $x$ to $z$, so that $d(x,z) = d(x,y_1)+d(y_1,y_2)+d(y_2,z)$.
  By the above observation, it suffices to show that this path does not duplicate any vertices.

  Since following the shortest path from $x$ to $y_1$ and then the shortest path from $y_1$ to $y_2$ does yield the shortest path from $x$ to $y_2$ (as $\btw{x}{y_1}{y_2}$), no vertices can be duplicated in this part of the path; and similarly no vertices can be duplicated in the part of the path from $y_1$ to $z$.
  So if any vertex were duplicated it would have to occur once strictly between $x$ and $y_1$ and again strictly between $y_2$ and $z$.
  Thus there is a path from this vertex to itself which visits $y_1$ and $y_2$ exactly once each, and since $y_1\neq y_2$ this path must contain a cycle, contradicting the assumption that the graph is a tree.

  (The magnitude homology of trees is calculated in~\cite[Corollary 31]{hw:mag-hom-gr}; it carries exactly the information of the number of vertices and edges.)
\end{eg}

\begin{eg}
  A complete graph also has no 4-cuts: since all nonzero distances are $1$, if $y_1\neq y_2$ then the hypotheses $\btw{x}{y_1}{y_2}$ and $\btw{y_1}{y_2}{z}$ imply $x=y_1$ and $y_2=z$, and the conclusion follows.
\end{eg}

The ``Menger-analogue'' of the uniqueness of geodesics is the following:

\begin{defn}\label{defn:geodetic}
  Two points $x,z\in X$ are \textbf{uniquely non-adjacent} if whenever $\btw x{y_1}z$ and $\btw{x}{y_2}{z}$, one of the following holds:
  \begin{itemize}
  \item $\btw x{y_1}{y_2}$ and $\btw{y_1}{y_2}z$.
  \item $\btw x{y_2}{y_1}$ and $\btw{y_2}{y_1}z$.
  \end{itemize}
  If any pair of distinct points is uniquely non-adjacent, we say that $X$
  is \textbf{geodetic}.
\end{defn}

Of course, $\R^n$ is geodetic, and geodeticity is inherited by subspaces.
The terminology is motivated by the following example:

\begin{eg}
  A connected graph with the shortest-path metric is geodetic in the above sense if and only if any two vertices are connected by a \emph{unique} shortest path (this is the usual meaning of ``geodetic'' in graph theory~\cite[p.104]{ore}).
  On one hand, if the latter holds, and $\btw x{y_1}z$ and $\btw{x}{y_2}{z}$, then $y_1$ and $y_2$ both lie on the unique shortest path from $x$ to $z$, hence their positions on that path can be compared.

  On the other hand, if $X$ is geodetic in the sense of \cref{defn:geodetic}, and $x$ and $z$ are connected by two shortest paths, let $y_1$ and $y_2$ be the first vertices after $x$ on the two paths.
  Then $\btw x{y_1}z$ and $\btw x{y_2}z$ (otherwise the paths would not be shortest); but $d(x,y_1)=d(x,y_2) = 1$, so $\btw x{y_1}{y_2}$ and $\btw x{y_2}{y_1}$ both imply $y_1=y_2$.
  By induction, the entire two shortest paths coincide.

  Every tree is geodetic, as is any cycle of odd length, any complete graph, and any block graph (one obtained by joining complete graphs together at vertices).
  But a cycle of even length is not: antipodal points thereon are not uniquely non-adjacent.
\end{eg}

\begin{thm}\label{thm:h2zero}
  Suppose that
  \begin{itemize}
  \item $X$ is geodetic; and
  \item either
    \begin{itemize}
    \item[--] $X$ is Menger convex and has no 4-cuts, or
    \item[--] $X$ is geodesic.
    \end{itemize}
  \end{itemize}
  Then $H_{2,*}(X)=0$.
\end{thm}
\begin{proof}
  Recall that the generating 2-chains in grading $\ell$ are triples $\chain{x_0,x_1,x_2}$ such that $x_0\neq x_1\neq x_2$ and $d(x_0,x_1)+d(x_1,x_2)=\ell$, and the boundary map $d_2$ takes such a triple to $-\chain{x_0,x_2}$ if $x_1$ is strictly between $x_0$ and $x_2$, and $0$ otherwise.
  Thus, the 2-cycles are finite linear combinations $\sum_{x,y,z} a_{xyz} \chain{x,y,z}$ such that for all $x,z$ we have
  \begin{equation}
    \sum_{\sbtw xyz} a_{xyz} = 0.\label{eq:h2-2cycles}
  \end{equation}
  We want to show that any such cycle is a boundary.
  The sum splits into two parts: those for which $y$ is between $x$ and $z$ and those for which it isn't; we will show that both are boundaries.

  To show that $\sum_{\sbtw xyz} a_{xyz} \chain{x,y,z}$ is a boundary, we use geodeticity.
  Because of~\eqref{eq:h2-2cycles}, it suffices to show that $\chain{x,y_1,z}- \chain{x,y_2,z}$ is a boundary whenever $y_1$ and $y_2$ are both between $x$ and $z$.
  By geodeticity, we have either
  \begin{align*}
    d(x,y_1)+d(y_1,y_2)+d(y_2,z) &=  d(x,z) \qquad \text{or}\\
    d(x,y_2)+d(y_2,y_1)+d(y_1,z) &=  d(x,z).
  \end{align*}
  In the first case, $d(\chain{x,y_1,y_2,z}) = \chain{x,y_1,z}- \chain{x,y_2,z}$, while in the second case $d(-\chain{x,y_2,y_1,z}) = \chain{x,y_1,z}- \chain{x,y_2,z}$.

  Now suppose $y$ is not between $x$ and $z$; here we use the second pair of assumptions.
  In the case when $X$ is Menger convex and has no 4-cuts, we can choose a $w$ with $\sbtw ywz$ by Menger convexity.
  If we had $\sbtw xyw$, then because $X$ has no 4-cuts we would have $\sbtw xyz$, a contradiction.
  Thus $y$ is not between $x$ and $w$, so $d_3(\chain{x,y,w,z}) = \chain{x,y,z}$ and hence $\chain{x,y,z}$ is a boundary.

  On the other hand, if instead $X$ is geodesic, let $a=d(y,z)$ and let $\gamma : [0,a] \to X$ be an isometry with $\gamma(0)=y$ and $\gamma(a)=z$.
  Suppose that $\btw xy{\gamma(t)}$ for all $t\in (0,a)$, i.e.\ that 
  \[ d(x,y) + d(y,\gamma(t)) = d(x,\gamma(t)) \]
  for all such $t$.
  Since $d(y,-)$ and $d(x,-)$ are continuous functions, and $\lim_{t\to a} \gamma(t) = z$, it follows that also
  \[ d(x,y) + d(y,z) = d(x,z), \]
  i.e.\ $\btw xyz$, a contradiction.
  Thus there exists some $t_0\in (0,a)$ such that $y$ is not between $x$ and $\gamma(t_0)$, whence $d_3(\chain{x,y,\gamma(t_0),z}) = \chain{x,y,z}$.
\end{proof}

\begin{cor}
  If $X$ is a closed convex subset of $\R^n$, then $H_{2,*}(X)=0$.\qed
\end{cor}

The presence of the two assumptions in \cref{thm:h2zero}, which are used in disjoint parts of the proof, suggests that there are two ways in which $H_{2,*}(X)$ can fail to vanish.
On the one hand, if $X$ is geodetic, then $H_{2,*}(X)$ detects some kind of
``failure of simultaneous convexity for triangles''.

\begin{thm}\label{thm:2cxvH2}
  If $X$ is geodetic and has no 4-cuts, then $H_{2,\ell}(X)$ is freely generated by the ordered triples $\chain{x,y,z}$ of distinct points such that $d(x,y)+d(y,z)=\ell$, $y$ is not between $x$ and $z$, $x$ and $y$ are adjacent, and $y$ and $z$ are adjacent.
\end{thm}
\begin{proof}
  The proof of \cref{thm:h2zero} shows that $\sum_{\sbtw xyz} a_{xyz} \chain{x,y,z}$ is a boundary, and that $\chain{x,y,z}$ is a boundary if $y$ is not between $x$ and $z$ and either $x$ and $y$ are non-adjacent or $y$ and $z$ are non-adjacent.
  Moreover, these boundaries generate the entire group of boundaries, since the boundary of a generating 3-chain $\chain{x,y,w,z}$ is either $0$, $\chain{x,y,z}$, $\chain{x,w,z}$, or $\chain{x,w,z}-\chain{x,y,z}$ according to whether $y$ is between $x$ and $w$ and whether $w$ is between $y$ and $z$.
  Thus, $H_{2,*}(X)$ is generated by what is left, which is what the theorem claims.
\end{proof}

\begin{eg}
  If $X$ is a closed annulus in the plane, then $H_{2,*}(X)$ is freely generated by the ordered triples $\chain{x,y,z}$ of distinct points all lying on the inner boundary of $X$.
\end{eg}

\begin{eg}
  If $X= X_1\sqcup X_2$ is the disjoint union of two convex sets in $\R^n$, then $H_{2,*}(X)$ is freely generated by the ordered triples $\chain{x,y,z}$ such that $x$ and $z$ lie in one component, $y$ lies in the other, and the segments $\overline{xy}$ and $\overline{yz}$ do not intersect $X$ except at their endpoints.
\end{eg}

On the other hand, and perhaps more interestingly, $H_{2,*}(X)$ can be
nonzero if $X$ is Menger convex but not geodetic.
In this case, $H_{2,*}(X)$ detects the ``failure of geodeticity'', which intuitively says something about whether pairs of points can be connected by multiple distinct geodesics.

\begin{eg}\label{thm:h2s1}
  Let $X=S^1$ with the geodesic metric (\emph{not} the subspace metric induced from $\R^2$), scaled so that the distance between two points is the angle between them.
  This is Menger convex, and indeed geodesic, so that $H_{1,*}(X)=0$.

  A point $y$ is between $x$ and $z$ exactly when it lies on the \emph{shorter} arc connecting $x$ and $z$.
  If $x$ and $z$ are antipodal, then every point $y$ is between $x$ and $z$.
  Moreover, of three distinct points $x,y,z$, either exactly one of them is between the other two, or none of them is between the other two.

  Since $X$ is geodesic, the second half of the proof of \cref{thm:h2zero} still applies.
  Thus it remains to consider the differences $\chain{x,y_1,z}-\chain{x,y_2,z}$ where $y_1$ and $y_2$ are strictly between $x$ and $z$.
  Moreover, although $X$ is not geodetic, it almost is: if $x$ and $z$ are not antipodal, then they \emph{are} uniquely non-adjacent.
  Thus, the proof of \cref{thm:h2zero} shows that $\chain{x,y_1,z}-\chain{x,y_2,z}$ is a boundary in this case.

  Moreover, if $x$ and $z$ \emph{are} antipodal, the same argument shows that $\chain{x,y_1,z}-\chain{x,y_2,z}$ is again a boundary if $y_1$ and $y_2$ lie in the same one of the two semicircles into which $x$ and $z$ disconnect $X$.
  Thus, what remain are the differences $\chain{x,y_1,z}-\chain{x,y_2,z}$ where $x$ and $z$ are antipodal, $y_1$ lies in one semicircle and $y_2$ lies in the other.
  The choice of $y_1$ and $y_2$ does not matter in homology (since changing them modifies the difference by a boundary), so we can consider each $\chain{x,y_1,z}-\chain{x,y_2,z}$ to be a single generator parametrized by the ordered pair of antipodal points $x,z$; or equivalently by a single point $x$, since $z$ is determined by $x$.
  (Switching $y_1$ and $y_2$ negates the generator, but we can make a consistent choice by, say, stipulating that the cyclic order $x \leadsto y_1 \leadsto z \leadsto y_2$ be counterclockwise.)
  Since antipodal points are always at distance $\pi$, we have
  \[ H_{2,\ell}(X) =
  \begin{cases}
    0 & \qquad \ell\neq\pi\\
    \Z[S^1] & \qquad \ell=\pi
  \end{cases}
  \]
  Intuitively, $H_{2,*}(X)$ is detecting the fact that antipodal points are connected by more than one distinct geodesic.
\end{eg}

For a completely general metric space, $H_{2,*}(X)$ can fail to vanish for a combination of these two reasons.
This is often the case for graphs with the shortest path metric, as studied in~\cite{hw:mag-hom-gr}: such spaces are never Menger convex, often have 4-cuts, and are often not geodetic.
In particular, the difference observed in ~\cite[\S A.1]{hw:mag-hom-gr} between the magnitude homology of odd and even cycle graphs should be partially explained by the fact that odd cycles are geodetic while even ones are not.

In all the examples above, the magnitude homology of a metric space is
torsion-free. But this is not the case in general, even for
graphs. Kaneta and Yoshinaga constructed a graph whose third homology has
torsion \cite[Corollary~5.12]{ky:maghom-ordcplx} (answering a question of
Hepworth and Willerton~\cite[\S1.2.2]{hw:mag-hom-gr}).  Sazdanovic and
Summers then showed that \emph{every} finitely generated abelian
group arises as a subgroup of the magnitude homology of some graph
\cite[Theorem~3.14]{ss:tor-mag-hom}.

\section{Magnitude homology of semicartesianly enriched categories}
\label{sec:hmag}

We now move on to define a more general notion of magnitude homology.
As we recalled in \cref{sec:mag}, magnitude can be defined for arbitrary $\V$-enriched categories, where $\V$ is a monoidal category equipped with a size.
Our current setting for magnitude homology is not as general; we require $\V$ to satisfy the following condition.

\begin{defn}
  A symmetric monoidal category $\V$ is \textbf{semicartesian} if its unit object $\done$ is the terminal object.
\end{defn}

\begin{egs}
  Of course, any cartesian monoidal category is semicartesian, such as $\mathrm{Set}$ or $\FinSet$.
  But $[0,\infty]$ (or $[0,\infty)$) is also semicartesian, since its unit object is $0$, even though it is not cartesian.
  The categorical cartesian product on $[0,\infty]$ is $\max$, so the monoidal category $[0,\infty]_{\max}$ from \cref{eg:size-ultra,eg:mag-ultra} is also semicartesian.
  We can also restrict to any full monoidal subcategory, such as $\mathbf{2} = \{0\le 1\}\subseteq \FinSet$ (whose enriched categories are preorders) or $\N_\infty = \N \cup \{\infty\} \subseteq [0,\infty]$ (whose enriched categories include graphs with the shortest-path metric).
\end{egs}

We also have to categorify the semiring $\k$ and the size function $\size :
\ob(\V) \to \k$.  We will replace $\k$ by a \emph{symmetric monoidal
  abelian category} $\A$, and $\size$ by a \emph{strong symmetric monoidal
  functor} $\Sigma : \V\to \A$.  That is, $\A$ is enriched over abelian
groups and has finite biproducts and well-behaved kernel-cokernel
factorizations, with a coherent tensor product $\otimes:\V\times \V\to\V$
that is additive in each argument, and $\Sigma$ is a functor with coherent
isomorphisms $\Sigma(V\otimes W) \cong \Sigma V \otimes \Sigma W$ and
$\Sigma\done\cong\done$. (See \cite[Chapters~VII and~VIII]{maclane:cwm} for
monoidal and abelian categories.)

\begin{eg}\label{eg:fab}
  When $\V=\FinSet$, we let $\A=\Ab$ be the category of abelian groups,
  with its usual tensor product, and $\Sigma(X)$ the free abelian group on $X$.
\end{eg}

\begin{eg}\label{eg:Rab}
  When $\V=[0,\infty)$, we let $\A=\prod_\R \Ab$ be the category of $\R$-graded abelian groups, with the convolution monoidal structure:
  \[ (A\otimes B)_\ell = \bigoplus_{j+k=\ell} A_j \otimes B_k \]
  and define $\Sigma: [0,\infty) \to \A$ by
  \[
    (\Sigma(\ell))_k =
    \begin{cases}
      \Z &\qquad k=\ell\\
      0 &\qquad \text{otherwise.}
    \end{cases}
  \]
  The action of $\Sigma$ on a morphism $\ell_1 > \ell_2$ is the only
  possibility, the zero map. 
\end{eg}

\begin{egs}\label{eg:subcat}
  Any such $\Sigma:\V\to\A$ can be restricted to any monoidal subcategory of $\V$.
  For instance, \cref{eg:fab} can be restricted to $\mathbf{2}\subseteq \FinSet$, and \cref{eg:Rab} can be restricted to $\N\subseteq [0,\infty)$.
  In the latter case, the restricted $\Sigma : \N \to \prod_\R \Ab$ lands in the monoidal subcategory $\prod_\Z \Ab$ (or even $\prod_\N \Ab$), so we can use that as the target.
\end{egs}

\begin{eg}\label{eg:maxab}
  When $\V=[0,\infty)_{\max}$, we can let $\A=\prod_\R \Ab$ as before, with the corresponding convolution monoidal structure:
  \[ (A\otimes B)_\ell = \bigoplus_{\max(j,k)=\ell} A_j \otimes B_k \]
  and the same functor $\Sigma$ as in \cref{eg:Rab}, which is also strong symmetric monoidal for these two different monoidal structures.
\end{eg}

Now we can give a general analogue of \cref{defn:mc}.

\begin{defn}\label{defn:gmc}
  Let $\V$ be a semicartesian symmetric monoidal category, $\A$ a closed symmetric monoidal abelian category, and $\Sigma :\V\to\A$ a strong symmetric monoidal functor.
  The \textbf{unnormalized magnitude complex} of a $\V$-category $X$ is a chain complex in $\A$ defined by
  \[ \uMCS_n(X) = \bigoplus_{x_0,\dots,x_n \in X} \Sigma X(x_0,x_1) \otimes \cdots \otimes \Sigma X(x_{n-1},x_n). \]
  Its boundary map $d_n : \uMCS_{n}(X) \to \uMCS_{n-1}(X)$ is an alternating sum
  \[ d_n = \sum_{i=0}^n (-1)^i d_n^i \]
  where $d_n^0$ and $d_n^n$ are induced, respectively, by discarding $x_0$ or $x_n$ from the indices and using the maps
  \[
    \Sigma X(x_0,x_1) \to \Sigma \done \cong \done \qquad\text{and}\qquad
    \Sigma X(x_{n-1},x_n) \to \Sigma \done \cong \done
  \]
  (which exist because $\done\in \V$ is terminal), while for $1\le i \le n-1$ the map $d^i_n$ is induced by discarding $x_i$ from the indices and using the composition map
  \[ \Sigma X(x_{i-1},x_i) \otimes \Sigma X(x_i,x_{i+1}) \cong
    \Sigma (X(x_{i-1},x_i) \otimes X(x_i,x_{i+1})) \to
    \Sigma X(x_{i-1},x_{i+1}). \]
  We leave it to the reader to check that $d\circ d = 0$.

  For each $0\le i\le n-1$ there is a degeneracy map $s_i : \uMCS_{n-1} \to \uMCS_n$ induced by duplicating $x_i$ in the indices and using the identities map
  \[ \done \cong \Sigma\done \to \Sigma X(x_i,x_i). \]
  The \textbf{(normalized) magnitude complex} is the induced chain complex $\MCS_*(X)$ where $\MCS_n(X)$ is the quotient of $\uMCS_n(X)$ by the images of all the degeneracies:
  \[ \MCS_n(X) = \uMCS_n(X) \Big/ \textstyle\bigcup_{i=0}^{n-1} \im(s_i) .\]
  The \textbf{magnitude homology} of $X$ is the homology of the chain complex $\MCS_*(X)$:
  \[ \hmag_n(X) =
    \ker\big(\MCS_{n}(X) \xrightarrow{d} \MCS_{n-1}(X)\big) \Big/
    \im\big(\MCS_{n+1}(X) \xrightarrow{d} \MCS_{n}(X)\big).
  \]
  Note that each $\hmag_n(X)$ is an object of $\A$.
\end{defn}

\begin{eg}
  When $\V=\FinSet$ and $\Sigma$ is the free abelian group functor, then $\uMCS_n(X)$ is the free abelian group generated the set of composable strings of $n$ morphisms in $X$, and $\MCS_n(X)$ is the free abelian group generated by the set of composable strings of $n$ \emph{nonidentity} morphisms in $X$.
  The pieces of the differential $d^i_n$ in $\uMCS_n(X)$ discard the first or last morphisms, or compose a pair of morphisms; while the differential in $\MCS_n(X)$ does the same except that if a composition yields an identity morphism then that part of the differential is defined to be $0$.
  In particular, $\MCS_*(X)$ is the normalized complex of simplicial chains in the nerve of $X$, so that $\hmag_*(X)$ is just the ordinary homology of (the geometric realization of) that nerve.

  Of course, since $\mathbf{2}\subseteq \FinSet$, this analysis applies also to $\mathbf{2}$-enriched categories, recovering the ordinary homology of the nerve of a preoder.
\end{eg}

\begin{eg}
  When $\V=[0,\infty)$ and $\Sigma$ is as in \cref{eg:Rab}, $\MCS_n(X)$ can be identified with the magnitude complex defined in \cref{defn:mc}.
  Specifically, $\uMCS_n(X)$ is the complex defined analogously to \cref{defn:mc} but without the restriction that each $x_i\neq x_{i+1}$, and the quotient by degeneracies kills all the generators in which some $x_i = x_{i+1}$.
  Hence $\hmag_n(X)$ is the magnitude homology defined in \cref{defn:hmag-met}:
  \[ \hmag_n(X) = \{ H_{n,\ell}(X) \}_\ell. \]
  Note that the grading $n$ is the homological one and the grading $\ell$ comes from the intrinsic $\R$-grading of the objects of $\A$.

  As observed in \cref{eg:subcat}, when $\V=\N$, we can take $\Sigma$ to land in $\prod_\N \Ab$.
  And as noted in \cref{sec:met}, when furthermore $X$ is a graph with the shortest path metric, $\MCS_n(X)$ coincides with the magnitude complex of~\cite{hw:mag-hom-gr}, so that $\hmag_*(X)$ is their magnitude homology of a graph.
\end{eg}

\begin{eg}\label{eg:ultra-maghom}
  When $\V=[0,\infty)_{\max}$ with $\Sigma$ as in \cref{eg:maxab}, we
    obtain an $\R$-graded \emph{ultramagnitude homology} theory for
    ultrametric spaces.
    (And also for quasi-pseudo-ultrametric spaces, but for simplicity we assume here symmetry and separatedness.)
  While leaving the detailed exploration of this theory for future research, we can make a few observations about it.

  The ultramagnitude chain group $\MC_{n,\ell}(X)$ is the free abelian group on tuples $\chain{x_0,\dots,x_n}$ where $\max(d(x_0,x_1),\dots,d(x_{n-1},x_n)) = \ell$ and each $x_i \neq x_{i+1}$.
  The boundary component $d_n^i$ discards the $i^{\mathrm{th}}$ point as long as this doesn't change the \emph{maximum} distance, which means in particular that compared to the ordinary magnitude homology of a metric space, $d^0_n$ and $d^n_n$ are no longer always zero.
  But we still have $d^0_1 = d^1_1 = 0$,
so $H_0(X)$ is again free on the points of $X$, and all 1-chains are cycles.

  The ultramagnitude homology $H_{1,\ell}(X)$ is thus again a quotient of the free abelian group generated by pairs $\chain{x,y}$ such that $d(x,y)=\ell$, but the relations imposed are different.
  (Assume $\ell>0$ for nontriviality.)
  Since in an ultrametric space all triangles are either acute isosceles (two equal sides and the third smaller) or equilateral, from a generating 2-chain $\chain{x,y,z}$ there are four possible kinds of relations we get in $H_{1,\ell}(X)$:
  \begin{enumerate}
  \item If $d(x,y)=d(y,z)=\ell$ but $d(x,z)<\ell$, then $\chain{x,y} + \chain{y,z} = 0$.\label{item:ur1}
  \item If $d(x,y)=d(x,z)=\ell$ but $d(y,z)<\ell$, then $\chain{x,y} - \chain{x,z} = 0$.\label{item:ur2}
  \item If $d(x,z)=d(y,z)=\ell$ but $d(x,y)<\ell$, then $\chain{y,z} - \chain{x,z} = 0$.\label{item:ur3}
  \item If $d(x,y)=d(y,z)=d(x,z)=\ell$, then $\chain{x,y} - \chain{x,z} + \chain{y,z} = 0$.\label{item:ur4}
  \end{enumerate}
  In particular, from~\ref{item:ur1} we get $\chain{x,y} = - \chain{y,x}$ (hence $\chain{x,x}=0$), and with this in hand~\ref{item:ur2} and~\ref{item:ur3} are interderivable.

  More concretely, $H_{1,\ell}(X)$ can be described as follows.\footnote{We thank the referee for this simplification of our original description.}
  Write $x\sim y$ if $d(x,y)\le \ell$, and $x\approx y$ if $d(x,y)< \ell$; in an ultrametric space these are equivalence relations, with $\approx$ refining $\sim$ (i.e.\ each $\sim$-equivalence class is a disjoint union of $\approx$-equivalence classes).
  Relations~\ref{item:ur2} and~\ref{item:ur3} say precisely that $\chain{x,y}$ depends only on the $\approx$-class of $x$ and $y$.
  Moreover, if $\mathbf{x}$ and $\mathbf{y}$ are $\approx$-classes, then there exist $x\in \mathbf{x}$ and $y\in \mathbf{y}$ such that $d(x,y)=\ell$ if and only if $d(x,y)=\ell$ for \emph{all} $x\in \mathbf{x}$ and $y\in \mathbf{y}$, and if and only if $\mathbf{x}$ and $\mathbf{y}$ are distinct and in the same $\sim$-class.

  Thus, we can equivalently consider $H_{1,\ell}(X)$ to be generated by pairs $\chain{\mathbf x,\mathbf y}$, where $\mathbf{x}$ and $\mathbf{y}$ are distinct $\approx$-classes in the same $\sim$-class, thereby automatically incorporating relations~\ref{item:ur2} and~\ref{item:ur3}.
  In terms of these generators,~\ref{item:ur1} says $\chain{\mathbf x,\mathbf y} = - \chain{\mathbf y,\mathbf x}$, while~\ref{item:ur4} says that $\chain{\mathbf x,\mathbf y} + \chain{\mathbf y,\mathbf z} = \chain{\mathbf x,\mathbf z}$ when $\mathbf{x},\mathbf{y},\mathbf{z}$ are pairwise distinct and in the same $\sim$-class.

  Even more simply, we can take the generators to be pairs $\chain{\mathbf x,\mathbf y}$ of arbitrary (not necessarily distinct) $\approx$-classes in the same $\sim$-class, modulo the single relation $\chain{\mathbf x,\mathbf y} + \chain{\mathbf y,\mathbf z} = \chain{\mathbf x,\mathbf z}$: this then implies $\chain{\mathbf x,\mathbf x}=0$ and $\chain{\mathbf x,\mathbf y} = - \chain{\mathbf y,\mathbf x}$.

  In fact $H_{1,\ell}(X)$ is a \emph{free} abelian group.
  There does not seem to be a canonical choice of basis, but if we choose one $\approx$-class in each $\sim$-class, then a basis is given by the generators $\chain{\mathbf x,\mathbf y}$ where $\mathbf{x}$ is the chosen $\approx$-class in its $\sim$-class and $\mathbf{y}$ is distinct from $\mathbf{x}$.
  The relations $\chain{\mathbf x,\mathbf y} = - \chain{\mathbf y,\mathbf x}$ and $\chain{\mathbf x,\mathbf y} + \chain{\mathbf y,\mathbf z} = \chain{\mathbf x,\mathbf z}$ then uniquely determine all the other generators as linear combinations of these basis elements.
  In particular, if $X$ is finite then the rank of $H_{1,\ell}(X)$ is $|X/\mathord\approx| - |X/\mathord\sim|$.
\end{eg}

\begin{rmk}
  A reader familiar with simplicial techniques may observe that $\MCS_*(X)$ is the normalized chain complex associated to a simplicial object in $\A$ by the Dold-Kan correspondence.
  This simplicial object is in fact a two-sided bar construction $B_*(\Sigma\done,\Sigma X,\Sigma\done)$ for the chain-complex-enriched category $\Sigma(X)$, which means that magnitude homology is in fact a particular kind of Hochschild homology.
  Standard results (e.g.~\cite[\S III.2]{goerss_jardine}) then imply that the magnitude homology can also be calculated using the unnormalized complex $\uMCS_*(X)$:
  \[ \hmag_n(X) \cong
    \ker\big(\uMCS_{n}(X) \xrightarrow{d} \uMCS_{n-1}(X)\big) \Big/
    \im\big(\uMCS_{n+1}(X) \xrightarrow{d} \uMCS_{n}(X)\big).
  \]
  This is related to the simplicial approach to magnitude homology of graphs described in~\cite[\S8]{hw:mag-hom-gr}.
  We will not need this level of abstraction, although it does suggest a potential generalization: we could allow $\Sigma$ to take values in arbitrary chain complexes, not necessarily concentrated in degree 0.
  The following invariance result can also be derived formally from simplicial tools, but we give an explicit proof.
\end{rmk}

\begin{thm}\label{thm:hmag-func}
  For any $\V$-functor $H:X\to X'$, there is an induced map
  \[ H_* : \hmag_*(X) \to \hmag_*(X'), \]
  which behaves functorially under composition.
  Moreover, if $K:X\to X'$ is another $\V$-functor and $\mu:H\to K$ a transformation, then $H_* = K_*$.
\end{thm}
\begin{proof}
  The first statement is straightforward. For the second, we begin by
  constructing a chain homotopy $m$ between the chain maps $\uMCS(X) \to
  \uMCS(X')$ induced by $H$ and $K$. For $x_0, \ldots, x_n \in X$ and $0
  \leq i \leq n$, there is a map
  \begin{multline*}
        X(x_0, x_1) \otimes\cdots\otimes X(x_{n - 1}, x_n) \\
    \longrightarrow
    X'(Px_0, Px_1) \otimes\cdots\otimes X'(Px_{i - 1}, PX_i)
    \otimes X'(Px_i, QX_i) \\
    \otimes
    X'(Qx_i, Qx_{i + 1}) \otimes\cdots\otimes X'(Qx_{n - 1}, Qx_n)
  \end{multline*}
  in $\V$ induced by the functorial actions of $P$ and $Q$ and by the
  $x_i$-component $\done \to X'(Px_i, Qx_i)$ of $\mu$. Applying $\Sigma$ to
  this map in $\V$ and summing over all $x_0, \ldots, x_n$ gives a map
  $m_n^i: \uMCS_n(X) \to \uMCS_{n + 1}(X)$ in $\A$. Put $m_n = \sum_{i =
    0}^n (-1)^i m_n^i$. A routine check shows that $m = \{m_n\}$ is a chain
  homotopy on unnormalized magnitude complexes, as claimed. Moreover, the
  maps $m_n$ preserve degeneracy of elements, and therefore also define a
  chain homotopy at the level of normalized complexes.  Hence, passing to
  homology, $H_* = K_*$.
\end{proof}

\begin{cor}\label{thm:hmag-adj-invar}
  If $X$ and $X'$ are $\V$-categories related by a $\V$-adjunction, then $\hmag_*(X) \cong \hmag_*(X')$.
  In particular, this is the case if $X$ and $X'$ are equivalent $\V$-categories.\qed
\end{cor}

\section{Ranks and Euler characteristics}
\label{sec:rk-chi}

Our main goal is to show that magnitude is the Euler characteristic of magnitude homology.
Since Euler characteristic is defined as an alternating sum of ranks, while our general notion of magnitude homology is parametrized over an arbitrary abelian category, we need an abstract notion of ``rank''.

\begin{defn}\label{defn:rank}
  Let $\A$ be an abelian category and $\k$ an abelian group.
  A \textbf{rank function} is a partial function $\rnk : \ob(\A) \pto \k$ such that $\rnk(0)=0$ and for any short exact sequence in $\A$:
  \[ 0 \to A \to B \to C \to 0 \]
  $\rnk(B)$ is defined if and only if both $\rnk(A)$ and $\rnk(C)$ are defined, and in this case
  \[ \rnk (B) = \rnk (A) + \rnk (C). \]
  If $\k$ is ordered, then a rank function is \textbf{positive} if $\rnk(A) \ge 0$ for all $A$.
  If $\A$ is symmetric monoidal and $\k$ is a ring, then a rank function is \textbf{multiplicative} if $\rnk (\done) = 1$, and whenever $\rnk (A)$ and $\rnk (B)$ are defined, so is $\rnk(A\otimes B)$ and $\rnk(A\otimes B) = \rnk (A) \cdot \rnk (B)$.
  We say an object $A\in\A$ is \textbf{finite} if $\rnk(A)$ is defined.
\end{defn}

In particular, rank is additive on binary direct sums: $\rnk (A\oplus B) = \rnk (A) + \rnk (B)$.
The following is obvious:

\begin{lem}\label{thm:chi-size}
  If $\rnk : \ob(\A) \pto \k$ is a multiplicative rank function and $\Sigma:\V\to\A$ is a strong symmetric monoidal functor taking values in finite objects, then the composite $\rnk\circ\Sigma:\ob(\V)\to\k$ is a size (\cref{defn:size}).\qed
\end{lem}

\begin{eg}
  If $\A=\Ab$, the usual $\Z$-valued rank of an abelian group (defined, for instance, as the dimension of the $\Q$-vector space $A \otimes \Q$) is a multiplicative rank function.
  When composed with the free abelian group functor $\Sigma:\FinSet\to \Ab$, this yields the cardinality $\size : \FinSet\to \Z$ as in \cref{eg:size-card}.
\end{eg}

\begin{eg}
  For any small set of objects $\mathbf{B}$ in an abelian category $\A$ that is closed under subobjects, quotients, and extensions, there
  is a universal rank function on $\A$ with domain $\mathbf{B}$. Its target
  abelian group is the K-theory group $K_0(\mathbf{B})$ of $\mathbf{B}$
  considered as a Quillen exact category.

  Note if $\mathbf{B}$ is closed under countable direct sums, the ``Eilenberg swindle'' implies that $K_0(\mathbf{B})$ is trivial: for any $A\in \mathbf{B}$, we have $\bigoplus_\omega A \cong A \oplus \bigoplus_\omega A$, hence $\rnk(A) = 0$.
  Thus some ``finiteness'' criterion is necessary to have a nontrivial rank function.
\end{eg}

\begin{eg}
  Let $\A=\prod_\R \Ab$ be the category of $\R$-graded abelian groups, with
  the convolution monoidal structure from \cref{eg:Rab}, and $\k =
  \Q\ser{q^\R}$ the ring of Novikov series
  (Definition~\ref{defn:novikov-series}).  Define an $\R$-graded abelian
  group $A$ to be \textbf{Novikov finite} if each abelian group $A_\ell$
  has finite rank and these ranks are left-finite.  In this case we define
  the \textbf{Novikov rank} of $A$ to be
  \[\rnk (A) = \sum_\ell (\rnk (A_\ell)) q^\ell \quad\in \Q\ser{q^\R}. \]
  This is a multiplicative rank function.
  When composed with the functor $\Sigma:[0,\infty) \to \A$ from
    \cref{eg:Rab}, which takes values in Novikov finite objects, the
    induced size $\size : [0,\infty) \to \Q\ser{q^\R}$ is the precisely the
      universal size from \cref{eg:genrat} (composed with the embedding
      $\Q(q^\R) \hookrightarrow \Q\ser{q^\R}$). (Recall the notion of
      universal size from Example~\ref{eg:genpoly}.)
  Of course, this rank function can be restricted to $\prod_\Z\Ab$, which contains the image of $\N \subseteq [0,\infty)$.
\end{eg}

\begin{eg}\label{eg:rk-ultra}
  If $\A=\prod_\R \Ab$ has the alternative convolution monoidal structure
  from \cref{eg:maxab} corresponding to $[0,\infty)_{\max}$, we can let
    $\k=\Z$, fix a $k\in [0,\infty)$, and put
  \[ \rnk_{\le k}(A) = \sum_{\ell \le k} \rnk(A_\ell) \]
  if this sum is finite.
  This is a multiplicative rank function, whose induced size $\size_{\le k} = \rnk_{\le k} \circ \Sigma$ is the so-named one from \cref{eg:size-ultra}.
  Similarly, we have a $\rnk_{<k}$ inducing the size $\size_{<k}$.
\end{eg}

The Euler characteristic of an $\N$-graded (or $\Z$-graded) object $\{A_n\}$ (such as a chain complex or the homology of a chain complex) should be the alternating sum of ranks.
In the classical case, this is only defined when the sum is finite, but to deal with the Novikov case we incorporate a topology.

\begin{defn}\label{defn:euler-char}
  Let $\A$ be an abelian category with a rank function $\rnk : \ob(\A) \pto \k$, where $\k$ is a topological ring.
  The \textbf{Euler characteristic} of a graded object $\{A_n\}_{n\in \N}$ is the infinite sum
  \begin{equation}
    \chi(A) = \sum_n (-1)^n \,\rnk (A_n)\label{eq:chi}
  \end{equation}
  if that converges in the topology of $\k$ (otherwise it is undefined).
  If $\k$ is an ordered topological ring, we say that $\chi(A)$
  \textbf{converges absolutely} if the above sum converges absolutely,
  i.e.\ the sum $\sum_n |\rnk(A_n) |$
  converges.
\end{defn}

As we now show, absolute convergence is automatic in the two cases of most interest to us.

\begin{lem}
  Let $\k$ be a ring with the discrete topology and $\sum_n a_n$ an infinite series in $\k$.
  The following are equivalent:
  \begin{enumerate}
  \item $\sum_n a_n$ converges.
  \item $a_n=0$ for all but finitely many $n$.
  \item (If $\k$ is ordered) $\sum_n a_n$ converges absolutely.\qed
  \end{enumerate}
\end{lem}

\begin{eg}
  When $\A=\Ab$ with the usual rank function, where $\Z$ has the discrete topology, a graded object has an Euler characteristic (which is then absolutely convergent) if and only if it has nonzero rank in only finitely many degrees, in which case we recover the usual Euler characteristic of a chain complex or its homology.
\end{eg}

\begin{eg}
  For $\A=\prod_\R \Ab$ with the rank function $\rnk_{\le k}$ from \cref{eg:rk-ultra}, a graded object has an (absolutely convergent) Euler characteristic if and only if there is a bound $N\in\N$ such that whenever $n>N$ and $\ell \le k$ we have $A_{n,\ell} = 0$.
\end{eg}

Recall from \cref{thm:novikov-properties} that the Novikov series field $\Q\ser{q^\R}$ is a complete valued field under the metric $d(a,b) = e^{-\nu(a-b)}$, where $\nu(a)$ is the smallest exponent with nonzero coefficient in $a$.

\begin{lem}\label{thm:novser}
  Let $\sum_n (\sum_\ell a_{n,\ell} q^\ell)$ be an infinite series whose terms lie in $\Q\ser{q^\R}$.
  The following are equivalent:
  \begin{enumerate}
  \item $\sum_n (\sum_\ell a_{n,\ell} q^\ell)$ converges in $\Q\ser{q^\R}$.\label{item:ns1}
  \item For any $k\in\R$ there exists an $N\in \N$ such that $a_{n,\ell} = 0$ for all $n>N$ and $\ell\le k$.\label{item:ns2}
  \item $\sum_n (\sum_\ell a_{n,\ell} q^\ell)$ converges absolutely in $\Q\ser{q^\R}$.\label{item:ns3}
  \end{enumerate}
  In this case, the sum of the series is
  \begin{equation}
    \sum_n \Big(\sum_\ell a_{n,\ell} q^\ell\Big) =
    \sum_\ell \Big( \sum_n a_{n,\ell}\Big) q^\ell\label{eq:novsum}
  \end{equation}
  with each inner sum $\sum_n a_{n,\ell}$ being finite by~\ref{item:ns2}.
\end{lem}
\begin{proof}
  A neighborhood basis of zero in $\Q\ser{q^\R}$ consists of, for each $k\in\R$, the set of all Novikov series $\sum_\ell b_\ell q^\ell$ such that $b_\ell =0$ for all $\ell<k$.
  From this~\ref{item:ns1}$\Leftrightarrow$\ref{item:ns2} follows directly, along with~\eqref{eq:novsum}.
  For~\ref{item:ns3}, note that if condition~\ref{item:ns2} is satisfied by a series then it is also satisfied by its termwise absolute value.
\end{proof}

In~\eqref{eq:novsum}, the $\sum_\ell$'s are formal sums (our notation for elements of $\Q\ser{q^\R}$), while the $\sum_n$'s denote actual (infinite) summations.

\begin{rmk}
  Note that the right-hand side of~\eqref{eq:novsum} exists as soon as each inner sum $\sum_{n} a_{n,\ell}$ is finite, which means that for any $\ell\in\R$ there exists an $N\in\N$ such that $a_{n,\ell}=0$ for all $n>N$.
  This is a strictly weaker condition than \cref{thm:novser}\ref{item:ns2}; consider for instance the series defined by
  \[a_{n,\ell}=
    \begin{cases}
      1 &\qquad \text{if } \ell = \frac1{n+1}\\
      0 &\qquad \text{otherwise}.
    \end{cases}
  \]
  Hence, existence of the right-hand side of~\eqref{eq:novsum} is not sufficient for convergence of the left-hand side.
\end{rmk}

\begin{eg}
  Let $\A=\prod_\R \Ab$ with the Novikov rank function, and $A$ a graded object of $\A$ such that each $A_n$ is Novikov finite.
  Then if $\chi(A)$ exists, it is absolutely convergent, and equals
  \begin{equation}
    \chi(A) = \sum_\ell \Big( \sum_n (-1)^n \,\rnk(A_{n,\ell})\Big) q^\ell.\label{eq:novikov-chi}
  \end{equation}
  with each inner sum $\sum_n (-1)^n \,\rnk(A_{n,\ell})$ being finite.
\end{eg}

In the discrete situation, it is a standard fact that the Euler characteristic of a chain complex can equivalently be computed from its homology.
In the topological case, this is still true as long as we have absolute convergence:

\begin{thm}\label{thm:chi-hom}
  Let $\A$ be an abelian category and $\rnk : \ob(\A)\pto \k$ a positive rank function, where $\k$ is an ordered topological ring, and let $A_*\in\Ch_\A$ be a chain complex in $\A$ and $H_*(A)$ its homology.
  If $\chi(A_*)$ converges absolutely, then so does $\chi(H_*(A))$, and $\chi(A_*) = \chi(H_*(A))$.
\end{thm}
\begin{proof}
  The assumption requires in particular that each $A_n$ is finite.
  Then we have short exact sequences
  \[ 0 \to \ker(d_{n}) \to A_n \to \im(d_n) \to 0 \]
  so that $\ker(d_n)$ and $\im(d_n)$ are finite and $\rnk(A_n) = \rnk(\ker(d_n)) + \rnk(\im(d_n))$.
  Since ranks are positive, $\rnk(\ker(d_n)) \le \rnk(A_n)$ and $\rnk(\im(d_n)) \le \rnk(A_n)$, hence $\sum_n (-1)^n \,\rnk(\ker(d_n))$ and $\sum_n (-1)^n \,\rnk (\im(d_n))$ also converge absolutely.

  Similarly, we have short exact sequences
  \[ 0 \to \im(d_{n+1}) \to \ker(d_n) \to H_n(A) \to 0 \]
  so that $H_n(A)$ is finite and $\rnk(H_n(A)) = \rnk(\ker(d_n)) - \rnk(\im(d_{n+1}))$.
  Therefore, since absolute convergence allows us to freely rearrange infinite sums, we have
  \begin{align*}
    \chi(A_*)
    &= \sum_n (-1)^n \, \rnk( A_n) \\
    &= \sum_n (-1)^n ( \rnk( \ker(d_n)) + \rnk(\im(d_n)))\\
    &= \sum_n (-1)^n \,\rnk( \ker(d_n)) + \sum_n (-1)^n \, \rnk(\im(d_n))\\
    &= \sum_n (-1)^n \,\rnk( \ker(d_n)) - \sum_n (-1)^n \, \rnk(\im(d_{n+1}))\\
    &= \sum_n (-1)^n ( \rnk( \ker(d_n)) - \rnk(\im(d_{n+1})))\\
    &= \sum_n (-1)^n\, \rnk( H_n(A)) \\
    &= \chi(H_*(A)).
  \end{align*}
\end{proof}

Therefore, as long as the Euler characteristic of the magnitude complex converges absolutely, it is equal to the Euler characteristic of the magnitude homology.
Thus, our strategy to relate magnitude homology to magnitude will be to show that the magnitude is equal to the Euler characteristic of the magnitude \emph{complex}.
We will proceed to this in the next section, but first we state a homotopical condition that we will need.

As is well-known, tensor products in abelian categories are not usually left exact, i.e.\ they do not preserve monomorphisms.
However, in our examples the values of the functor $\Sigma$ are usually projective (even free), and tensor products with projective objects are usually left exact (i.e.\ projective objects are flat).
The condition we need is just a ``relative'' version of this flatness.

\begin{defn}\label{defn:qmon}
  Let $\A$ be a closed symmetric monoidal abelian category.
  A \textbf{cofibration} in $\A$ is a monomorphism whose cokernel is projective.
  We say $\A$ is \textbf{Quillen monoidal} if whenever $A\to B$ and $C\to
  D$ are cofibrations, so is the induced map from the pushout (called a
  \emph{pushout product}):
\setlength{\mathsurround}{0pt} 
  \[
    \begin{tikzcd}
      A\otimes C \ar[r] \ar[d] & B\otimes C \ar[d] \ar[ddr] \\
      A\otimes D \ar[r] \ar[drr] & \bullet \ar[ul,phantom,near start,"\ulcorner"] \ar[dr,dashed] \\
      & & B\otimes D
    \end{tikzcd}
  \]
\setlength{\mathsurround}{0.8pt} 
and moreover the unit object $\done$ is projective.
\end{defn}

Readers familiar with model category theory will recognize this as the ``shadow'' in $\A$ of an assertion that the projective Quillen model structure on $\Ch_\A$ is monoidal.

Note that an object $A$ is projective if and only if it is \textbf{cofibrant}, i.e.\ the map $0\to A$ is a cofibration.
Thus, if $\A$ is Quillen monoidal, then projective objects are closed under tensor products, and tensoring with a projective object preserves cofibrations.

Let $\V$ be a symmetric monoidal category and $\Sigma: \V \to \A$ a strong
symmetric monoidal functor. Then any $\V$-category $X$ gives rise to an
$\A$-category $\Sigma X$, with the same objects as $X$ and hom-objects
$(\Sigma X)(x, x') = \Sigma(X(x, x'))$.

\begin{defn}
  We say that $\Sigma X$ is \textbf{cofibrant} if each object $\Sigma X(x,x')$ is projective and each identities map $\done \to \Sigma X(x,x')$ is a cofibration.
\end{defn}

\begin{eg}
  The category $\Ab$ of abelian groups is Quillen monoidal. Take the
  $\Sigma$ of \cref{eg:fab} and an ordinary category $X$.
  Since free abelian groups are projective, $\Sigma X$ is always cofibrant.
\end{eg}

\begin{eg}
  The category $\prod_\R \Ab$ of $\R$-graded abelian groups is also Quillen monoidal.
  To see this, note that its limits, colimits, and cofibrations are pointwise, so we must prove that if $A\to B$ and $C\to D$ are cofibrations then so is each induced map
  \setlength{\mathsurround}{0pt} 
  \[
    \begin{tikzcd}
      \bigoplus_{j+k=\ell} A_j\otimes C_k \ar[r] \ar[d] & \bigoplus_{j+k=\ell} B_j\otimes C_k \ar[d] \ar[ddr] \\
      \bigoplus_{j+k=\ell} A_j\otimes D_k \ar[r] \ar[drr] & \bullet \ar[ul,phantom,near start,"\ulcorner"] \ar[dr,dashed] \\
      & & \bigoplus_{j+k=\ell} B_j\otimes D_k.
    \end{tikzcd}
  \]
  \setlength{\mathsurround}{0.8pt} 
  But since direct sums preserve pushouts, this follows from the
  corresponding fact for $\Ab$ applied to the maps $A_j\to B_j$ and $C_k\to
  D_k$. Take the $\Sigma$ of \cref{eg:Rab} and a metric space $X$.
  Again, since degreewise free objects are projective, $\Sigma X$ is always cofibrant.
\end{eg}

\section{Magnitude homology categorifies magnitude}
\label{sec:categ}

We now prove, under suitable assumptions, that magnitude is the Euler characteristic of magnitude homology.
For all of this section, we place ourselves in the following context.

\begin{assume}\label{assumption}
  Let $\V$ be a semicartesian symmetric monoidal category, $\A$ a closed symmetric Quillen monoidal abelian category, $\Sigma :\V\to\A$ a strong symmetric monoidal functor, $\k$ an ordered topological ring, $\rnk : \ob(\A) \pto \k$ a positive multiplicative rank function such that $\Sigma$ takes finite values, $\size = \rnk \circ \Sigma$ the induced size function, and $X$ a $\V$-category with finitely many objects such that $\Sigma X$ is cofibrant.
\end{assume}

Recall from \cref{defn:gmc} that the magnitude chain complex is defined by
\[ \MCS_n(x) = \uMCS_n(X) \Big/ \textstyle\bigcup_{i=0}^{n-1} \im(s_i) .\]
where $\uMCS_*(X)$ is the unnormalized magnitude complex and $s_i : \uMCS_{n-1}(X) \to \uMCS_n(X)$ are the degeneracy maps.
The rank of $\uMCS_n(X)$ is easy to compute:

\begin{lem}
  We have
  \[ \rnk\,\uMCS_n(X) = \sum_{x_0,\dots,x_n \in X} (\size X(x_0,x_1)) \cdots (\size X(x_{n-1},x_n)). \]
\end{lem}
\begin{proof}
  This follows immediately from the definition of $\uMCS$ in \cref{defn:gmc}, the facts that $\rnk$ is additive on $\oplus$ and multiplicative on $\otimes$, and the definition of $\size$.
\end{proof}

\begin{eg}
  If $\V=\FinSet$, then $\rnk\,\uMCS_n(X)$ is the number of composable strings of $n$ morphisms in $X$.
\end{eg}

The problem now is to deal with the quotient in the definition of $\MCS_n(X)$.
For objects $x_0,\dots,x_n$ of $X$ we define
\[ \uMCS_n(x_0,\dots,x_n) = \Sigma X(x_0,x_1) \otimes \cdots \otimes \Sigma X(x_{n-1},x_n). \]
Thus, according to \cref{defn:gmc}, we have
\[ \uMCS_n(X) = \bigoplus_{x_0,\dots,x_n} \uMCS_n(x_0,\dots,x_n). \]
The degeneracy map $s_i : \uMCS_{n-1}(X) \to \uMCS_{n}(X)$ sends each $\uMCS_{n-1}(x_0,\dots,x_{n-1})$ into $\uMCS_n(x_0,\dots,x_i,x_i,\dots,x_{n-1})$ via the identities map $\done\cong \Sigma\done \to \Sigma X(x_i,x_i)$.
Thus, if we define
\[ \uMDS_n(x_0,\dots,x_n) = \bigoplus_{i \text{ such that}\atop x_i = x_{i+1}} \uMCS_{n-1}(x_0,\dots,x_i,x_{i+2},\dots,x_n) \]
then the degeneracy maps whose image lands in $\uMCS_n(x_0,\dots,x_n)$ assemble into a single map
\begin{equation*}
  \stil_n : \uMDS_n(x_0,\dots,x_n) \to \uMCS_n(x_0,\dots,x_n).
\end{equation*}
In particular, note that if $x_i\neq x_{i+1}$ for all $i$, then $\uMDS_n(x_0,\dots,x_n) = 0$.
If we denote the cokernel of $\stil_n$ by $\MCS_n(x_0,\dots,x_n)$, we have
\[ \MCS_n(X) = \bigoplus_{x_0,\dots,x_n} \MCS_n(x_0,\dots,x_n). \]
Thus, it will suffice to compute the ranks of the objects $\MCS_n(x_0,\dots,x_n)$.

Note that we have an exact sequence
\[ \uMDS_n(x_0,\dots,x_n) \to \uMCS_n(x_0,\dots,x_n) \to \MCS_n(x_0,\dots,x_n) \to 0 \]
but the left-hand map is not injective.
Our strategy will be to build up the image of this map in an explicit way such that we can compute its rank.
(The reader familiar with simplicial techniques will recognize this as an
inductive construction of the ``latching object'' of a Reedy cofibrant simplicial diagram, see e.g.~\cite[Chapter 5]{hovey:modelcats} or~\cite[Chapter 15]{hirschhorn:modelcats}.)
We define, inductively, a family of objects $\MDS_n(x_0,\dots,x_n)$ with maps
\[ j_n : \MDS_n(x_0,\dots,x_n) \to \uMCS_n(x_0,\dots,x_n)\]
as follows.
When $n=0$, we set $\MDS_0(x_0) = 0$, the zero object, and $j_0$ the unique morphism $0\to \done$.  If $n>0$, we split into cases:
\begin{enumerate}
\item If $x_0 \neq x_1$, define
  \[ \MDS_n(x_0,\dots,x_n) = \Sigma X(x_0,x_1) \otimes \MDS_{n-1}(x_1,\dots,x_n). \]
  The map $j_n$ is
  \begin{multline}
    \Sigma X(x_0,x_1) \otimes \MDS_{n-1}(x_1,\dots,x_n)
    \\\xrightarrow{1 \otimes j_{n-1}} \Sigma X(x_0,x_1) \otimes \uMCS_{n-1}(x_1,\dots,x_n)
    = \uMCS_n(x_0,\dots,x_n).\label{eq:jn}
  \end{multline}
\item If $x_0=x_1$, let $\MDS_n(x_0,\dots,x_n)$ be the pushout
  \setlength{\mathsurround}{0pt} 
  \begin{equation}
    \begin{tikzcd}
      \MDS_{n-1}(x_1,\dots,x_n) \ar[r] \ar[d,"j_{n-1}"] &
      \Sigma X(x_1,x_1) \otimes \MDS_{n-1}(x_1,\dots,x_n) \ar[d] \\
      \uMCS_{n-1}(x_1,\dots,x_n) \ar[r] &
      \MDS_n(x_0,\dots,x_n) \ar[ul,phantom,near start,"\ulcorner"]
    \end{tikzcd}\label{eq:jnp}
  \end{equation}
  \setlength{\mathsurround}{0.8pt} 
  where the top map inserts identities in $\Sigma X(x_1,x_1)$.
  The map $j_n$ is induced by the universal property of the pushout from the maps~\eqref{eq:jn} and the insertion of identities
  \begin{equation*}
    \uMCS_{n-1}(x_1,\dots,x_n)
    \to \Sigma X(x_1,x_1) \otimes \uMCS_{n-1}(x_1,\dots,x_n)
    = \uMCS_n(x_0,\dots,x_n).
  \end{equation*}  
\end{enumerate}

\begin{lem}\label{thm:jnisacof}
  Each map $j_n$ is a cofibration, and in particular a monomorphism.
\end{lem}
\begin{proof}
  By induction, using the fact that $\A$ is Quillen monoidal.
  In the base case, $0\to \done$ is a cofibration because $\done$ is projective.
  And in both cases of the inductive step, $j_n$ is a pushout product of the cofibration $j_{n-1}$ with another cofibration, namely $0 \to \Sigma X(x_0,x_1)$ or $\done \to \Sigma X(x_1,x_1)$; thus it is also a cofibration.
\end{proof}

\begin{lem}\label{thm:jnimg}
  $j_n$ is isomorphic to the image of $\stil_n$.
\end{lem}
\begin{proof}
  We first note that $\uMDS_n(x_0,\dots,x_n)$ also has an inductive description.
  Namely, $\uMDS_0(x_0) = 0$, while for $n>0$ if $x_0\neq x_1$ then
  \[ \uMDS_n(x_0,\dots,x_n) = \Sigma X(x_0,x_1) \otimes \uMDS_{n-1}(x_1,\dots,x_n) \]
  while if $x_0=x_1$ then
  \[ \uMDS_n(x_0,\dots,x_n) = \big(\Sigma X(x_0,x_1) \otimes \uMDS_{n-1}(x_1,\dots,x_n)\big) \oplus \uMCS_{n-1}(x_1,\dots,x_n). \]

  We define inductively a surjection $p_n : \uMDS_n(x_0,\dots,x_n) \to \MDS_n(x_0,\dots,x_n)$ such that $j_n\circ p_n = \stil_n$; since image factorizations are unique up to isomorphism in an abelian category this will prove the result.
  When $n=0$, $p_0$ is just the identity $0\to 0$.
  When $n>0$, we have by induction a surjection $p_{n-1} : \uMDS_{n-1}(x_1,\dots,x_n) \to \MDS_{n-1}(x_1,\dots,x_n)$.
  In the case when $x_0\neq x_1$, we define $p_n$ to be
  \[ \Sigma X(x_0,x_1) \otimes \uMDS_{n-1}(x_1,\dots,x_n) \xrightarrow{1\otimes p_{n-1}} \Sigma X(x_0,x_1) \otimes \MDS_{n-1}(x_1,\dots,x_n). \]
  This is a surjection since $p_{n-1}$ is and since $\otimes$ preserves surjections in each variable (since $\A$ is closed monoidal), and satisfies $j_n\circ p_n = \stil_n$.
  In the other case when $x_0=x_1$, we define $p_n$ to be the composite
  \begin{multline*}
    \big(\Sigma X(x_0,x_1) \otimes \uMDS_{n-1}(x_1,\dots,x_n)\big) \oplus \uMCS_{n-1}(x_1,\dots,x_n) \\
    \xrightarrow{(1\otimes p_{n-1}) \oplus 1}
    \big(\Sigma X(x_0,x_1) \otimes \MDS_{n-1}(x_1,\dots,x_n)\big) \oplus \uMCS_{n-1}(x_1,\dots,x_n) \\
    \longrightarrow \MDS_{n}(x_0,\dots,x_n)
  \end{multline*}
  where the second map is the copairing of the two coprojections into the pushout~\eqref{eq:jnp}.
  This is a surjection for any pushout, while the first factor is a surjection as before because $p_{n-1}$ is.
  And again it is straightforward to check that $j_n\circ p_n = \stil_n$.
\end{proof}

\begin{lem}\label{thm:jnrk}
  Each object $\MDS_n(x_0,\dots,x_n)$ is finite, and its rank can be computed inductively as follows.
  \begin{enumerate}
  \item If $x_0\neq x_1$, then $\rnk\, \MDS_n(x_0,\dots,x_n)$ is
    \[  \size X(x_0,x_1) \cdot \rnk\, \MDS_{n-1}(x_1,\dots,x_n) \]
  \item If $x_0=x_1$, then $\rnk\, \MDS_n(x_0,\dots,x_n)$ is
    \[ 
      (\size X(x_0,x_1) - 1) \cdot \rnk\, \MDS_{n-1}(x_1,\dots,x_n) + \rnk\, \uMCS_{n-1}(x_1,\dots,x_n)
    \]
  \end{enumerate}
\end{lem}
\begin{proof}
  When $x_0\neq x_1$, this is clear from the definition and the fact that $\rnk$ is multiplicative.
  When $x_0=x_1$, we reformulate the pushout~\eqref{eq:jnp} as a short exact sequence:
  \begin{multline*}
    0\to \MDS_{n-1}(x_1,\dots,x_n) \to \uMCS_{n-1}(x_1,\dots,x_n) \oplus (\Sigma X(x_1,x_1) \otimes \MDS_{n-1}(x_1,\dots,x_n))
    \\ \to \MDS_n(x_0,\dots,x_n) \to 0.
  \end{multline*}
  Here exactness at $\MDS_{n-1}(x_1,\dots,x_n)$, i.e.\ injectivity of the map with this domain, follows from injectivity of $j_n$, which follows from \cref{thm:jnisacof}.
  Now the claim follows from additivity of $\rnk$ on short exact sequences.
\end{proof}

\begin{thm}
  Each object $\MCS_n(X)$ is finite, and
  \[ \rnk\,\MCS_n(X) = \sum_{x_0,\dots,x_n} \prod_{0\le i \le n-1} \Big(\size X(x_i,x_{i+1}) - \delta_{x_i,x_{i+1}}\Big) \]
  where the $\delta$'s are Kronecker's, $\delta_{x,x'} = 1$ if $x=x'$ and $0$ otherwise.
\end{thm}
\begin{proof}
  Of course, it suffices to show that
  \begin{equation}
    \rnk\,\MCS_n(x_0,\dots,x_n) = \prod_{0\le i \le n-1} \Big(\size X(x_i,x_{i+1}) - \delta_{x_i,x_{i+1}}\Big).\label{eq:rkmcs}
  \end{equation}
  We can re-express~\eqref{eq:rkmcs} in a more inductive manner as
  \begin{align*}
    \rnk\,\MCS_0(x_0) &= 1\\
    \rnk\,\MCS_n(x_0,\dots,x_n) &= (\size X(x_0,x_1) - \delta_{x_0,x_1}) \cdot \rnk\, \MCS_{n-1}(x_1,\dots,x_n).
  \end{align*}
  The first equation holds since $\MCS_n(x_0) = \done$, whose rank is $1$ by assumption.
  For the second, note that by \cref{thm:jnimg}, we have short exact sequences
  \[ 0\to \MDS_n(x_0,\dots,x_n) \to \uMCS_n(x_0,\dots,x_n) \to \MCS_n(x_0,\dots,x_n) \to 0, \]
  so $\rnk\,\MCS_n(x_0,\dots,x_n) = \rnk\,\uMCS_n(x_0,\dots,x_n) - \rnk\, \MDS_n(x_0,\dots,x_n)$.
  Now we divide into cases based on whether $x_0=x_1$, using \cref{thm:jnrk}.
  If $x_0\neq x_1$, then
  \begin{align*}
    \MoveEqLeft \rnk\,\MCS_n(x_0,\dots,x_n)\\
    &= \rnk\,\uMCS_n(x_0,\dots,x_n) - \rnk\, \MDS_n(x_0,\dots,x_n)\\
    &
    \!\begin{multlined}[t][9cm]
      =\size X(x_0,x_1) \cdot \rnk\,\uMCS_{n-1}(x_1,\dots,x_n) - \size X(x_0,x_1) \cdot \rnk\, \MDS_{n-1}(x_1,\dots,x_n)
    \end{multlined}\\
    &= \size X(x_0,x_1) \cdot \big( \rnk\,\uMCS_{n-1}(x_1,\dots,x_n) - \rnk\, \MDS_{n-1}(x_1,\dots,x_n)\big)\\
    &= \size X(x_0,x_1) \cdot \rnk\, \MCS_{n-1}(x_1,\dots,x_n).
  \end{align*}
  If $x_0=x_1$, then
  \begin{align*}
    \MoveEqLeft \rnk\,\MCS_n(x_0,\dots,x_n)\\
    &= \rnk\,\uMCS_n(x_0,\dots,x_n) - \rnk\, \MDS_n(x_0,\dots,x_n)\\
    &\!
    \begin{multlined}[t][11cm]
      = \size X(x_0,x_1) \cdot \rnk\,\uMCS_{n-1}(x_1,\dots,x_n)\\
      {}- (\size X(x_0,x_1) - 1) \cdot \rnk\, \MDS_{n-1}(x_1,\dots,x_n)
      {}- \rnk\, \uMCS_{n-1}(x_1,\dots,x_n)
    \end{multlined}\\
    &= (\size X(x_0,x_1) - 1) \cdot \big( \rnk\,\uMCS_{n-1}(x_1,\dots,x_n) - \rnk\, \MDS_{n-1}(x_1,\dots,x_n)\big)\\
    &= (\size X(x_0,x_1) - 1) \cdot \rnk\, \MCS_{n-1}(x_1,\dots,x_n).
  \end{align*}
\end{proof}

It follows that if the following infinite sum converges absolutely, then it is equal to the Euler characteristic of magnitude homology:
\[ \chi(\MCS_*(X)) =  \sum_n (-1)^n \sum_{x_0,\dots,x_n} \prod_{0\le i \le n-1} \Big(\size X(x_i,x_{i+1}) - \delta_{x_i,x_{i+1}}\Big). \]
Of course, this is automatically the case if this sum is finite:

\begin{eg}
  If $\V=\FinSet$, then $\size X(x_i,x_{i+1}) - \delta_{x_i,x_{i+1}}$ is the number of \emph{nonidentity} morphisms from $x_i$ to $x_{i+1}$.
  Thus, in this case $\rnk\,\MCS_n(X)$ is the number of composable strings of $n$ nonidentity morphisms in $X$, or equivalently the number of nondegenerate $n$-simplices in the nerve of $X$.
  Thus, if this nerve has only finitely many simplices, then $\chi(\MCS_*(X))$ converges absolutely to the Euler characteristic of magnitude homology.
  Of course, the alternating sum of the number of nondegenerate $n$-simplices is also the simplicial Euler characteristic of the nerve of $X$, so we recover~\cite[Proposition 2.11]{leinster:ec-cat}.
\end{eg}

For metric spaces (even finite ones) the sum is generally infinite, but it nevertheless always converges absolutely:

\begin{thm}
  For any finite quasi-metric space $X$, the Euler characteristic $\chi(\MC_{*,*}(X))$ of the magnitude complex converges absolutely in $\Q\ser{q^\R}$.
  Hence by \cref{thm:chi-hom}, it equals the Euler characteristic $\chi(H_{*,*}(X))$ of the magnitude homology.
\end{thm}
\begin{proof}
  Since $X$ is finite and separated, there is an $\epsilon>0$ such that $d(x,y)\ge\epsilon$ for any $x\neq y$.
  Using \cref{thm:novser}\ref{item:ns2}, let $k\in \R$ be given, and choose $N\in\N$ such that $N \epsilon > k$; we show that if $n>N$ then the following sum contains no terms $q^\ell$ for $\ell\le k$:
  \[ \sum_{x_0,\dots,x_n} \prod_{0\le i \le n-1} \Big(\size X(x_i,x_{i+1}) - \delta_{x_i,x_{i+1}}\Big) \]
  Since $\size X(x_i,x_{i+1}) = q^{d(x_i,x_{i+1})}$, which equals 1 if and only if $x_i=x_{i+1}$, this sum is equivalently
  \[ \sum_{x_0\neq \dots\neq x_n} \prod_{0\le i \le n-1} q^{d(x_i,x_{i+1})} = \sum_{x_0\neq \dots\neq x_n} q^{d(x_0,x_1) + \cdots + d(x_{n-1},x_n)}. \]
  But each term in this sum has exponent
  \[ d(x_0,x_1) + \cdots + d(x_{n-1},x_n) > n\epsilon > N \epsilon > k \]
  which is what we wanted.
\end{proof}

\begin{eg}
  Unfortunately, the Euler characteristic of the ultramagnitude complex of an ultrametric space with respect to the ranks $\rnk_{\le k}$ (or $\rnk_{<k}$) from \cref{eg:rk-ultra} never converges except in trivial cases.
  As soon as $X$ contains two points $x\neq y$ with $d(x,y)\le k$, then $\MC_{n,k}(X)$ is nonzero for all $n$, as it contains a generator $\chain{x,y,x,y,\dots}$.
  Thus, $\chi(\MC_{*,*}(X))$ fails to converge (in the discrete topology on $\Z$).
  Even using a topological rank function similar to the Novikov rank would not help, since the generators $\chain{x,y,x,y,\dots}$ all lie in the same graded piece $\MC_{*,d(x,y)}(X)$.
  It is an open question whether any kind of ultramagnitude homology can be related to any kind of ultramagnitude.
\end{eg}

Returning to the case of general $\V$, it remains to identify $\chi(\MCS_*(X))$ with the definition of magnitude in \cref{sec:mag} using M\"{o}bius inversion and the matrix $Z_X$.
Since this is the main theorem of the paper, we repeat the standing \cref{assumption}.

\begin{thm}\label{thm:hmag}
  Let $\V$ be a semicartesian symmetric monoidal category, $\A$ a closed symmetric Quillen monoidal abelian category, $\Sigma :\V\to\A$ a strong symmetric monoidal functor, $\k$ an ordered topological ring, $\rnk : \ob(\A) \pto \k$ a positive multiplicative rank function such that $\Sigma$ takes finite values, $\size = \rnk \circ \Sigma$ the induced size function, and $X$ a $\V$-category with finitely many objects such that $\Sigma X$ is cofibrant.

  If the Euler characteristic $\chi(\MCS_*(X))$ is absolutely convergent, then $X$ has M\"{o}bius inversion, and
  \[ \mg[\size]{X}=\chi(\MCS_*(X)) = \chi(\hmag_*(X)).\]
\end{thm}
\begin{proof}
  The idea is to ``expand $Z_X^{-1}$ as a geometric series'' via the following non-rigorous calculation:
  \begin{equation*}
    \maybe{ Z_X^{-1} = \frac{1}{Z_X} = \frac{1}{1+(Z_X-1)}
      = \sum_{n=0}^\infty (-1)^n (Z_X-1)^n }
  \end{equation*}
  where by $1$ in $(Z_X-1)$ we mean the identity matrix.
  To make rigorous sense of this, we start by observing that by the definition of matrix multiplication, we have
  \[ (Z_X-1)^n(x_0,x_n) = \sum_{x_1,\dots,x_{n-1}} (Z_X-1)(x_0,x_1)\cdot \cdots\cdot (Z_X-1)(x_{n-1},x_n) \]
  Here $(Z_X-1)(x_0,x_1)$ means the $(x_0,x_1)$-entry of the matrix $Z_X-1$, and so on.
  On the other hand, we also have
  \[ (Z_X-1)(x,x') =
  \begin{cases}
    \size X(x,x') &\qquad x\neq x'\\
    \size X(x,x') - 1 &\qquad x= x'.
  \end{cases}
  \]
  Thus, by~\eqref{eq:rkmcs}, we have
  \[(Z_X-1)^n(x_0,x_n) = \sum_{x_1,\dots,x_{n-1}} \rnk\,\MCS_n(x_0,\dots,x_n).\]
  Since ranks are positive, this sum is between $0$ and $\rnk\, \MCS_n(X)$.
  Since $\chi(\MCS_*(X)) = \sum_n (-1)^n \rnk\, \MCS_n(X)$ converges absolutely, the series
  \[ \sum_{n=0}^\infty (-1)^n (Z_X-1)^n \]
  converges absolutely in the induced (entry-wise) topology of matrices over $\k$.

  Since multiplication by a single number distributes over convergent series, and convergent infinite sums can be interchanged with finite sums, we can also distribute \emph{matrix} multiplication over convergent series of matrices.
  Thus we have
  \allowdisplaybreaks
  \begin{align*}
    Z_X \cdot \sum_{n=0}^\infty (-1)^n (Z_X-1)^n
    &= (1 + (Z_X-1)) \cdot \sum_{n=0}^\infty (-1)^n (Z_X-1)^n\\
    &= \left(\sum_{n=0}^\infty (-1)^n (Z_X-1)^n\right) + (Z_X-1) \left(\sum_{n=0}^\infty (-1)^n (Z_X-1)^n\right)\\
    &= \left(\sum_{n=0}^\infty (-1)^n (Z_X-1)^n\right) + \left(\sum_{n=0}^\infty (-1)^n (Z_X-1)^{n+1}\right)\\
    &= \left(\sum_{n=0}^\infty (-1)^n (Z_X-1)^n\right) + \left(\sum_{n=1}^\infty (-1)^{n-1} (Z_X-1)^{n}\right)\\
    &= (-1)^0 (Z_X-1)^0 + \sum_{n=1}^\infty ((-1)^n + (-1)^{n-1})(Z_X-1)^n \\
    &= 1 + 0 = 1.
  \end{align*}
  A similar argument shows that $\sum_{n=0}^\infty (-1)^n (Z_X-1)^n \cdot Z_X = 1$.
  Thus, $Z_X$ is invertible, so $X$ has M\"{o}bius inversion.
  Moreover, we have a formula for its inverse:
  \[Z_X^{-1} = \sum_{n=0}^\infty (-1)^n (Z_X-1)^n.\]
  Thus, if $\sm$ denotes summing all the entries of a matrix, we have
  \begin{align*}
    \sm(Z_X^{-1})
    &= \sum_{n=0}^\infty (-1)^n \sm((Z_X-1)^n)\\
    &= \sum_{n=0}^\infty (-1)^n \sum_{x_0,x_n} (Z_X-1)^n(x_0,x_n)\\
    &= \sum_{n=0}^\infty (-1)^n \sum_{x_0,x_n} \sum_{x_1,\dots,x_{n-1}} \rnk\,\MCS_n(x_0,\dots,x_n)\\
    &= \sum_{n=0}^\infty (-1)^n\, \rnk\,\MCS_n(X)\\
    &= \chi(\MCS_*(X)).
  \end{align*}
  As noted above, by \cref{thm:chi-hom} this can also be identified with $\chi(\hmag_*(X))$.
\end{proof}

Recall that in \cref{thm:mag-met} we showed that any finite quasi-metric space has a magnitude valued in the ring $\Q(q^\R)$ of generalized rational functions, which can be embedded into $\Q\ser{q^\R}$.
Thus we obtain \cref{thm:magcat-met}:

\begin{cor}\label{thm:magcat-met-2}
  If $X$ is a finite quasi-metric space, then
  \begin{align*}
    \mg[\Q\ser{q^\R}]{X} &= \sum_{n=0}^\infty (-1)^n \Big( \sum_{\ell\in\R} \rnk\, H_{n,\ell}(X) \, q^\ell \Big)\\
    &= \sum_{\ell\in\R} \Big( \sum_{n=0}^\infty (-1)^n \,\rnk\, H_{n,\ell}(X)\Big) q^\ell,
  \end{align*}
  the former infinite sum converging in the topology of $\Q\ser{q^\R}$.\qed
\end{cor}

\cref{thm:hmag} shows that if the Euler characteristic of magnitude homology converges, then it is equal to the magnitude (which \textit{a fortiori} exists).
On the other hand, if we only know that the magnitude exists, it doesn't necessarily follow that the Euler characteristic of magnitude homology converges; but we can nevertheless recover the magnitude using one of the standard methods for summing a divergent series.

\begin{thm}\label{thm:hmag-divser}
  In the context of \cref{assumption}, suppose that $X$ has M\"{o}bius inversion relative to $\size = \chi\circ \Sigma$.
  Then the formal power series
  \[ \sum_{n=0}^\infty \rnk\,\MCS_n(X)\cdot t^n \quad \in \k\pser{t} \]
  is equal to a rational function (i.e.\ its image in $\k\ser{t}$ is also in the image of $\k(t)$), and evaluating this rational function at $t=-1$ yields $\mg[\k]{X}$.
\end{thm}
\begin{proof}
  Give $\k\pser{t}$ the usual topology in which the set of multiples of
  $t^n$ is a generating neighborhood of $0$, for each $n \in \N$. Every
  formal power series $\sum_{n=0}^\infty a_n t^n$ (where $a_n\in \k$)
  converges to itself in this topology. 
  Thus, similar arguments as in \cref{thm:hmag} (and~\cite[Lemma 2.1]{bl:ec-div}) show that 
  \[(1-(Z_X-1)t) \cdot \sum_{n=0}^\infty (Z_X-1)^n t^n = 1 \]
  and likewise on the opposite side.
  It follows that $(1-(Z_X-1)t)$ is invertible over $\k\pser{t}$, with inverse $\sum_{n=0}^\infty (Z_X-1)^n t^n$.
  In fact it is obvious that $(1-(Z_X-1)t)$ is invertible over $\k\ser{t}$,
  since its determinant $\det(1-(Z_X-1)t)$ is a polynomial whose constant
  term (its value at $t=0$) is $1$.
  Thus we also have the usual formula for its inverse in $\k\ser{t}$, which is therefore equal to the inverse we have just computed in $\k\pser{t}$:
  \[ \sum_{n=0}^\infty (Z_X-1)^n t^n = \frac{\adj(1-(Z_X-1)t)}{\det(1-(Z_X-1)t)}. \]
  Hence, summing the entries of both sides, we get
  \begin{align*}
    \sum_{n=0}^\infty  \rnk\,\MCS_n(X)\cdot t^n 
    &= \sm\left(\sum_{n=0}^\infty (Z_X-1)^n t^n\right)\\
    &= \frac{\sm(\adj(1-(Z_X-1)t))}{\det(1-(Z_X-1)t)}
  \end{align*}
  which evidently lies in $\k(t)$.
  Finally, since $Z_X$ is assumed invertible over $\k$, when we evaluate this rational function at $t=-1$ we get the sum of the entries of the usual formula for its inverse there, hence $\mg X$.
\end{proof}

\begin{rmk}
  We have seen in \cref{thm:hmag-adj-invar} that magnitude homology is invariant under equivalence (and even adjunction) of categories.
  However, the property of having M\"{o}bius inversion, and the (stronger, by \cref{thm:hmag}) property that $\chi(\MCS_*(X))$ converges absolutely, are not so invariant.

  Indeed, as we have noted before, a $\V$-category with M\"{o}bius inversion must be skeletal, since two isomorphic objects would produce two identical rows in $Z_X$.
  Similarly, the numbers $\rnk\,\MCS_n(x_0,\dots,x_n)$ are not invariant under equivalence: equivalent $\V$-categories $X\simeq X'$ have homotopy equivalent magnitude complexes, but the ranks of the individual chain groups will generally differ.
  Moreover, although one can sum the divergent series of \cref{thm:hmag-divser} in more generality to define a more general notion of ``series magnitude'' for enriched categories, as in~\cite{bl:ec-div}, the result is no longer invariant under equivalence of categories.
\end{rmk}

We can however extend the result of \cref{thm:hmag} slightly while remaining invariant under equivalence.

\begin{cor}
  Suppose $X$ and $X'$ are equivalent $\V$-categories that both satisfy \cref{assumption}, and that $\chi(\MCS_*(X'))$ is absolutely convergent.
  Then $X$ has magnitude in the sense of \cref{defn:mag2}, and $\mg{X} = \chi(H_*(X))$.
\end{cor}
\begin{proof}
  By \cref{thm:hmag}, $X'$ has M\"{o}bius inversion and $\mg{X'}=\chi(H_*(X'))$.
  But $H_*(X) \cong H_*(X')$, so they have the same Euler characteristic.
  And by \cref{thm:mag-eqv0,thm:mag-eqv}, $\mg X = \mg {X'}$.
\end{proof}

\begin{rmk}
It is natural to wonder about stronger sorts of invariance, for instance under Morita equivalence of $\V$-categories, or equivalently Cauchy completion (cocompletion under absolute colimits).
Indeed, magnitude homology is Morita-invariant if $\Sigma$ is cocontinuous; this can be deduced from its identification with a kind of Hochschild homology, or proven directly.

However, magnitude is not Morita-invariant: for instance, as noted in~\cite{leinster:ec-cat}, the free ordinary category on an idempotent has magnitude $\frac12$, whereas its Cauchy-completion has magnitude $1$.
Note that \cref{thm:hmag} does not apply to this example, since the nerve of an idempotent has infinitely many nondegenerate simplices, so $\chi(\MCS_*(X))$ does not converge.
On the other hand, \cref{thm:hmag} does apply both to a metric space and its Cauchy completion, but in this case the functor $\Sigma$ from \cref{eg:Rab} is not cocontinuous.
We do not know any examples of a cocontinuous $\Sigma$ for which \cref{thm:hmag} can apply to both sides of a nontrivial Morita equivalence, but if there were one it would follow that the corresponding numerical magnitude was also Morita-invariant.
\end{rmk}

\begin{rmk}\label{rmk:non-semicart}
  It is also natural to wonder whether the semicartesianness assumption on $\V$ can be relaxed.
  We know of one example suggesting that this may be possible.
  In~\cite{ckl:mag-fdalg} it is shown that if $\V$ is the category of finite-dimensional vector spaces over an algebraically closed field $K$, with $\size:\ob(\V)\to\Q$ the dimension, and $X$ is the $\V$-category of indecomposable projective modules over a $K$-algebra $A$ of finite dimension and finite global dimension, then $\mg X = \sum_n (-1)^n \,\mathrm{dim} \, \operatorname{Ext}^n_A(S,S)$, where $S$ is the direct sum of a system of representatives of isomorphism classes of simple $A$-modules.
  This clearly looks like the Euler characteristic of a (co)homology theory, and indeed because $K$ is a field we can identify $\operatorname{Ext}^n_A(S,S)$ with the Hochschild \emph{cohomology} $\HH^n(A;\hom_K(S,S))$, which is also a Hochschild cohomology of $X$ since $X$ is Morita equivalent to $A$.
  However, it is unclear to us how this could be unified with our
  \cref{thm:hmag}, or why the dualization to cohomology appears here. Note
  that there is also a theory of magnitude cohomology, introduced by
  Hepworth~\cite{hepworth:mag-cohom}.
\end{rmk}

\section{Open problems}
\label{sec:open-problems}

There are many open problems regarding magnitude homology of general enriched categories, including the following.
\begin{enumerate}
\item We have already mentioned in \cref{rmk:non-semicart} the question of whether \cref{thm:hmag} can be generalized to the non-semicartesian case.
\item What other enriching categories $\V$ support an interesting magnitude
  homology?
\item K\"{u}nneth and Mayer--Vietoris theorems for magnitude homology were
  proved for graphs in~\cite{hw:mag-hom-gr}, and extended to metric spaces
  in~\cite{bk:mag-hom-diag}.  Can they be generalized to arbitrary enriched
  categories?
\end{enumerate}
There are also many open problems regarding magnitude homology of metric spaces specifically, such as the following.
\begin{enumerate}[resume]
\item What can be said about the geometric meaning of $H_{n,\ell}(X)$ for
  $n>2$? Progress has been made since the appearance of the preprint
  version of this paper; see the works cited at the start of
  section~\ref{sec:geo-met}, as well as~\cite{ay:geom-app} for the case of
  graphs. But much remains to be understood.
\item Our theorem relating magnitude homology to magnitude applies only to \emph{finite} metric spaces.
  Magnitude homology is defined for arbitrary metric spaces; indeed this is one of the virtues of an algebraic invariant over a numerical one, and we have seen in \cref{sec:geo-met} that it detects interesting information when applied directly to infinite metric spaces.
  On the other hand, the \emph{magnitude} of an infinite metric space can also be defined as the limit of approximating finite subspaces, or directly by ``replacing sums with integrals'' in the definition of the magnitude of finite metric spaces~\cite{meckes:posdef}.
  Can this generalized notion of magnitude also be recovered from the magnitude homology?
\item A related observation is that when magnitude homology groups of an infinite metric space are nonzero, they tend to be infinitely (even uncountably) generated.
  However, their generators tend to be points or tuples of points of $X$, which suggests that they could be endowed with some ``topological'' structure to make them more manangeable (e.g.\ they could be generated by a set that is compact or has finite measure).
  Some such structure might be necessary to calculate a finite ``size'' in order to determine the magnitude of an infinite metric space from its magnitude homology.
\item Magnitude homology only ``notices'' whether the triangle inequality
  is a strict equality or not. Is there a ``blurred'' version that notices
  ``approximate equalities''? And relatedly, almost everyone who
  encounters both magnitude homology and persistent homology feels that
  there should be some relationship between them. What is it? These
  questions have been addressed by Otter~\cite{otter}, who first related
  the two theories using a notion of ``blurred magnitude homology'', by Govc and
  Hepworth~\cite{gh:pers-mag}, who showed that the magnitude of a finite
  metric space is the ``persistent magnitude'' of its blurred magnitude
  homology, and by Cho~\cite{cho}, who provided a common framework for
  persistent and magnitude homology. What more can be said about the
  relationship?
\end{enumerate}

\bibliographystyle{gtart}

\begin{thebibliography}{CKL16}

\bibitem[AI20]{ay:geom-app}
Yasuhiko Asao and Kengo Izumihara.
\newblock Geometric approach to graph magnitude homology.
\newblock arXiv:2003.08058, 2020.

\bibitem[Asa19]{asao:maghom-geodesic}
Yasuhiko Asao.
\newblock Magnitude homology of geodesic metric spaces with an upper curvature
  bound.
\newblock arXiv:1903.11794, 2019.

\bibitem[BC18]{bc:mag-cpteuc}
Juan~Antonio Barcel{\'o} and Anthony Carbery.
\newblock On the magnitudes of compact sets in {E}uclidean spaces.
\newblock {\em American Journal of Mathematics}, 140(2):449--494, 2018.

\bibitem[BK20]{bk:mag-hom-diag}
R{\'e}mi Bottinelli and Tom Kaiser.
\newblock Magnitude homology, diagonality, medianness, {K}{\"u}nneth and
  {M}ayer--{V}ietoris.
\newblock arXiv:2003.09271, 2020.

\bibitem[BL08]{bl:ec-div}
Clemens Berger and Tom Leinster.
\newblock The {E}uler characteristic of a category as the sum of a divergent
  series.
\newblock {\em Homology Homotopy Appl.}, 10(1):41--51, 2008.

\bibitem[Cho19]{cho}
Simon Cho.
\newblock Quantales, persistence, and magnitude homology.
\newblock arXiv:1910.02905, 2019.

\bibitem[CKL16]{ckl:mag-fdalg}
Joseph Chuang, Alastair King, and Tom Leinster.
\newblock On the magnitude of a finite dimensional algebra.
\newblock {\em Theory and Applications of Categories}, 31(3):63--72, 2016.

\bibitem[GG17]{gg:mag-domeuc}
Heiko Gimperlein and Magnus Goffeng.
\newblock On the magnitude function of domains in {E}uclidean space.
\newblock arXiv:1706.06839. {\it American Journal of Mathematics}, in press,
  2017.

\bibitem[GH19]{gh:pers-mag}
Dejan Govc and Richard Hepworth.
\newblock Persistent magnitude.
\newblock arXiv:1911.11016. {\it Journal of Pure and Applied Algebra}, in
  press, 2019.

\bibitem[GJ99]{goerss_jardine}
Paul~G. Goerss and John~F. Jardine.
\newblock {\em Simplicial Homotopy Theory}.
\newblock Birkh\"{a}user, 1999.

\bibitem[Gom19]{gomi:maghom-geodesic}
Kiyonori Gomi.
\newblock Magnitude homology of geodesic space.
\newblock arXiv:1902.07044, 2019.

\bibitem[Gom20]{gomi:magcplx}
Kiyonori Gomi.
\newblock Smoothness filtration of the magnitude complex.
\newblock {\em Forum Mathematicum}, 32(3):625--639, 2020.

\bibitem[Gu18]{gu:grmaghom-algmorse}
Yuzhou Gu.
\newblock Graph magnitude homology via algebraic {M}orse theory.
\newblock arXiv:1809.07240, 2018.

\bibitem[Hep18]{hepworth:mag-cohom}
Richard Hepworth.
\newblock Magnitude cohomology.
\newblock arXiv:1807.06832, 2018.

\bibitem[Hir03]{hirschhorn:modelcats}
Philip~S. Hirschhorn.
\newblock {\em Model Categories and their Localizations}, volume~99 of {\em
  Mathematical Surveys and Monographs}.
\newblock American Mathematical Society, 2003.

\bibitem[HJ12]{hj:matrix-analysis}
Roger~A. Horn and Charles~R. Johnson.
\newblock {\em Matrix Analysis}.
\newblock Cambridge University Press, Cambridge, 2nd edition, 2012.

\bibitem[Hov99]{hovey:modelcats}
Mark Hovey.
\newblock {\em Model Categories}, volume~63 of {\em Mathematical Surveys and
  Monographs}.
\newblock American Mathematical Society, 1999.

\bibitem[HW15]{hw:mag-hom-gr}
Richard Hepworth and Simon Willerton.
\newblock Categorifying the magnitude of a graph.
\newblock arXiv:1505.04125, 2015.

\bibitem[Jub18]{jubin:maghom}
Beno{\^\i}t Jubin.
\newblock On the magnitude homology of metric spaces.
\newblock arXiv:1803.05062, 2018.

\bibitem[Kel82]{kelly:basic}
G.~Max Kelly.
\newblock {\em Basic Concepts of Enriched Category Theory}, volume~64 of {\em
  London Mathematical Society Lecture Note Series}.
\newblock Cambridge University Press, Cambridge, 1982.
\newblock Also \emph{Reprints in Theory and Applications of Categories} 10
  (2005), 1--136.

\bibitem[KY18]{ky:maghom-ordcplx}
Ryuki Kaneta and Masahiko Yoshinaga.
\newblock Magnitude homology of metric spaces and order complexes.
\newblock arXiv:1803.04247, 2018.

\bibitem[Law74]{lawvere:metric-spaces}
F.~William Lawvere.
\newblock Metric spaces, generalized logic, and closed categories.
\newblock {\em Rend. Sem. Mat. Fis. Milano}, 43:135--166, 1974.
\newblock Reprinted as Repr. Theory Appl. Categ. 1:1--37, 2002.

\bibitem[Lei08]{leinster:ec-cat}
Tom Leinster.
\newblock The {E}uler characteristic of a category.
\newblock {\em Documenta Mathematica}, 13:21--49, 2008.

\bibitem[Lei13]{leinster:magnitude}
Tom Leinster.
\newblock The magnitude of metric spaces.
\newblock {\em Documenta Mathematica}, 18:857--905, 2013.

\bibitem[LM17]{lm:mag-survey}
Tom Leinster and Mark Meckes.
\newblock The magnitude of a metric space: from category theory to geometric
  measure theory.
\newblock In Nicola Gigli, editor, {\em Measure Theory in Non-Smooth Spaces}.
  de Gruyter Open, 2017.
\newblock arxiv:1606.00095.

\bibitem[Mec13]{meckes:posdef}
Mark~W. Meckes.
\newblock Positive definite metric spaces.
\newblock {\em Positivity}, 17(3):733--757, Sep 2013.

\bibitem[Mec15]{meckes:mag-div}
Mark~W. Meckes.
\newblock Magnitude, diversity, capacities, and dimensions of metric spaces.
\newblock {\em Potential Analysis}, 42(2):549--572, Feb 2015.

\bibitem[ML71]{maclane:cwm}
Saunders Mac~Lane.
\newblock {\em Categories for the Working Mathematician}.
\newblock Graduate Texts in Mathematics 5. Springer, New York, 1971.

\bibitem[Nov81]{novikov:multivalued}
Sergei~P. Novikov.
\newblock Multivalued functions and functionals. {A}n analogue of the {M}orse
  theory.
\newblock {\em Doklady Akademii Nauk SSSR}, 260(1):31--35, 1981.

\bibitem[ora17]{2444419}
orangeskid.
\newblock Positive definite matrix over a non-{Archimedean} field.
\newblock Mathematics Stack Exchange, 2017.
\newblock \url{https://math.stackexchange.com/q/2444419} (version: 2017-09-25).

\bibitem[Ore62]{ore}
Oystein Ore.
\newblock {\em Theory of Graphs}, volume~38 of {\em Colloquium Publications}.
\newblock American Mathematical Society, Providence, Rhode Island, 1962.

\bibitem[Ott18]{otter}
Nina Otter.
\newblock Magnitude meets persistence. {H}omology theories for filtered
  simplicial sets.
\newblock arXiv:1807.01540, 2018.

\bibitem[Pap05]{papadopoulos:metricspaces}
Athanase Papadopoulos.
\newblock {\em Metric Spaces, Convexity, and Nonpositive Curvature}.
\newblock European Mathematical Society, 2005.

\bibitem[SS19]{ss:tor-mag-hom}
Radmila Sazdanovic and Victor Summers.
\newblock Torsion in the magnitude homology of graphs.
\newblock arXiv:1912.13483, 2019.

\end{thebibliography}

\end{document}